\documentclass{amsart}
\usepackage{amsfonts}
\usepackage{amsmath,amsthm}
\usepackage{amsfonts,amssymb}
\usepackage{mathrsfs}
\usepackage{enumerate}

\usepackage{multirow}
\usepackage{booktabs}

\newtheorem{thm}{Theorem}[section]
\newtheorem{cor}[thm]{Corollary}
\newtheorem{lem}[thm]{Lemma}

\theoremstyle{remark}
\newtheorem{rem}[thm]{Remark}

\numberwithin{equation}{section}

\newcommand{\al}{\alpha}

\def\vz{\varepsilon}
\def\oz{\omega}
\def\lz{\lambda}
\def\Lz{\Lambda}

\def\dz{\delta}

\def\({\Bigl(}
\def \){ \Bigr)}

 \def\RR{{\mathbb R}}

\def\va{\varepsilon}

\begin{document}
\def\RR{\mathbb{R}}
\def\Exp{\text{Exp}}
\def\FF{\mathcal{F}_\al}

\title[] {On the power of standard information for tractability for
$L_2$-approximation in the randomized setting}

\author{Wanting Lu} \address{ School of Mathematical Sciences, Capital Normal
University, Beijing 100048,
 China.}
 \email{luwanting1234@163.com}

\author{Heping Wang} \address{ School of Mathematical Sciences, Capital Normal
University,
Beijing 100048,
 China.}
\email{wanghp@cnu.edu.cn}


\begin{abstract} We study approximation of multivariate functions from a  separable
 Hilbert space in the randomized setting with
the error measured in the weighted $L_2$ norm. We consider
algorithms that use standard information $\Lz^{\rm std}$
consisting of function values or general linear information
$\Lz^{\rm all}$ consisting of arbitrary linear functionals. We use
 the weighted  least squares regression algorithm to  obtain the
 upper estimates of the
minimal randomized error using  $\Lz^{\rm std}$.    We investigate
the equivalences of various notions of algebraic
 and exponential tractability  for
$\Lz^{\rm std}$  and $\Lz^{\rm all}$  for the normalized or
absolute  error criterion. We show that in the randomized setting
for the normalized or absolute  error criterion, the power of
$\Lz^{\rm std}$  is the same as that of $\Lz^{\rm all}$  for all
notions of exponential  and algebraic tractability without any
condition. Specifically, we  solve four  Open Problems  98,
100-102 as posed by E.Novak and H.Wo\'zniakowski in the book:
Tractability of Multivariate Problems, Volume III: Standard
Information for Operators, EMS Tracts in Mathematics, Z\"urich,
2012.

\end{abstract}

\maketitle
\input amssym.def

\section{Introduction}
We study  multivariate approximation ${\rm APP}=\{{\rm
APP}_d\}_{d\in\Bbb N}$, where
$${{\rm APP}}_d: F_d\to G_d\ \ {\rm with}\ \  {\rm APP}_d\,
f=f $$ is the compact embedding operator,  $F_d$ is  a separable
Hilbert function space on $D_d$, $G_d$ is a weighted $L_2$ space
on $D_d$,  $D_d\subset \Bbb R^d$, and the dimension $d$ is large
or even huge. We consider algorithms that use finitely many
information evaluations. Here information evaluation means  linear
functional on $F_d$ (general linear information) or function value
at some point (standard information). We use $\Lz^{\rm all}$ and
$\Lz^{\rm std}$ to denote the extended class of all linear
functionals (not necessarily continuous) and the extended class of
all function values (defined only almost everywhere),
respectively.

 For a given error threshold
$\vz\in(0,1)$, the information complexity $n(\vz, d)$ is defined
to be the minimal number of information evaluations for which the
approximation error of some algorithm is at most $\vz$.
Tractability is aimed at studying how the information complexity
$n(\vz,d)$ depends on $\vz$ and $d$. There are two kinds of
tractability based on polynomial convergence and exponential
convergence. The algebraic tractability (ALG-tractability)
describes how the information complexity $n(\vz, d)$ behaves as a
function of $d$ and $\vz^{-1}$,  while the exponential
tractability (EXP-tractability)
 does as one of $d$ and $(1+\ln\vz^{-1})$. The existing notions of
tractability mainly include strong polynomial tractability (SPT),
polynomial tractability (PT), quasi-polynomial tractability (QPT),
weak tractability (WT), $(s, t)$-weak tractability ($(s,t)$-WT),
and uniform weak tractability (UWT).
 In
recent years   the study of algebraic  and
 exponential tractability has attracted much interest, and a great number of
interesting results
 have been obtained
  (see
\cite{NW1, NW2, NW3, W, GW, X1, S1, SiW,  DLPW,  DKPW, KW, PP, X2,
IKPW, CW, LX2, PPXD} and the references therein).

This paper is devoted to investigating  the equivalences of
various notions of algebraic and exponential tractability  for
 $\Lz^{\rm std}$ and $\Lz^{\rm all}$
 in the randomized  setting (see \cite[Chapter
22]{NW3}). The class $\Lz^{\rm std}$ is much smaller and much more
practical,  and is much more difficult to analyze than the class
$\Lz^{\rm all}$. Hence, it is very important to study the power of
$\Lz^{\rm std}$ compared to $\Lz^{\rm all}$. There are many paper
devoted to this field (see \cite{NW3, WW07, KWW, K19,  HWW, LZ,
X5, WW01, KWW1,  NW16, NW17, KU, KU2, NSU, HKNPU, HKNPU1, KS}).

In \cite{WW07, NW3} the authors obtained the equivalences of
ALG-SPT, ALG-PT, ALG-QPT, ALG-WT  in the randomized setting for
 $\Lz^{\rm std}$ and $\Lz^{\rm all}$ for the normalized error
criterion  without any condition. Meanwhile, for the absolute
error criterion under some conditions, the equivalences of
ALG-SPT, ALG-PT, ALG-QPT, ALG-WT were also obtained in \cite{NW3}.

In this paper, we obtain the remaining equivalences of all notions
of algebraic and exponential  tractability   in the randomized
setting for  $\Lz^{\rm std}$ and $\Lz^{\rm all}$  for the
normalized or absolute error criterion  without any condition. Our
results particularly imply that for the absolute error criterion
the imposed conditions  are not necessary. This solves Open
Problems 98, 101, 102 as posed by Novak and Wo\'zniakowski in
\cite{NW3}. We also give an almost complete solution to  Open
Problem 100 in \cite{NW3}.

This paper is organized as follows.  Section 2 contains 5
subsections. In Subsections 2.1 and 2.2 we introduce the
approximation problem in the worst case and randomized settings.
The various notions of algebraic and exponential tractability are
given in Subsection 2.3. Subsection 2.4 is devoted to give the
equivalences of tractability for $\Lz^{\rm all}$ for the absolute
or normalized error criterion in the worst case and randomized
settings.
   Our main results, Theorems
  2.2, 2.3, 2.5, and 2.6
 are stated  in Subsection 2.5. In  Section 3, we give the proofs of Theorems 2.2 and
2.3. After that, in Section 4,  we establish the equivalence
results for the notions of algebraic tractability. The equivalence
results for the notions of exponential tractability are proved in
Section 5.

\section{Preliminaries and Main Results}

\subsection{Deterministic worst case setting}\

 For $d\in \Bbb N$, let $F_d$ be a separable Hilbert space
of $d$-variate functions defined on $D_d\subset \Bbb R^d$,
$G_d=L_{2}(D_d,\rho_{d}(x)dx)$ be a weighted $L_2$ space, where
$D_d$ is a Borel measurable subset of $\mathbb{R}^{d}$ with
positive Lebesgue measure, $\rho_d$ is a probability density
function on $D_d$. We consider the multivariate approximation
problem ${\rm APP}=\{{\rm APP}_d\}_{d\in \Bbb N}$  in the
deterministic worst case setting which is defined via the compact
embedding operator
\begin{equation}\label{2.1} {{\rm APP}}_d: F_d\to G_d\ \ {\rm with}\ \  {\rm APP}_d\,
f=f.\end{equation} We approximate ${\rm APP}_df$ by algorithms
$A_{n,d}f$ of the form
\begin{equation}\label{2.2}A_{n,d}f=\phi
_{n,d}(L_1(f),L_2(f),\dots,L_n(f)),\end{equation} where
$L_1,L_2,\dots,L_n$ are general linear functionals  on $F_d$,
 and $\phi _{n,d}:\;\Bbb R^n\to
G_d$ is an arbitrary measurable mapping. The worst case
approximation error for the algorithm $A_{n,d}$ of the form
\eqref{2.2} is defined as
 $$e^{\rm wor}(A_{n,d})=\sup_{f\in F_d,\ \|f\|_{F_d}\le1}
\|{\rm APP}_d\,f-A_{n,d}f\|_{G_d}.$$ The $n$th minimal worst case
error  is defined by
$$e^{\rm wor}(n,d;\Lz^{\rm all})=\inf_{A_{n,d} \ {\rm with}\ L_i\in \Lz^{\rm all}}e^{\rm wor}(A_{n,d}),$$
where the infimum is taken over all algorithms of the form
\eqref{2.2}.

For $n=0$, we use $A_{0,d}=0$. We obtain  the so-called initial
error $e^{\rm wor}(0,d; \Lz^{\rm all})$, defined by
$$e^{\rm wor}(0,d;\Lz^{\rm all} )=\sup_{f\in F_d,\ \|f\|_{F_d}\le1}\|{\rm APP}_d\,
f\|_{G_d}.$$

From \cite{NW1, HNW} we know  that $e^{\rm wor}(n,d;\Lz^{\rm
all})$ depends on the eigenpairs
$\big\{(\lz_{k,d},e_{k,d})\big\}_{k=1}^\infty$ of the operator
$$W_d={\rm APP}_d^*\,{\rm APP}_d \colon F_d\to F_d
,$$where ${\rm APP}_d$ is given by \eqref{2.1}, ${\rm APP}_d^*$ is
the adjoint operator of ${\rm APP}_d$, and $$\lz_{1,d}\ge
\lz_{2,d}\ge \dots \lz_{n,d}\dots \ge 0.$$
 That is,   $\{e_{j,d}\}_{j\in \Bbb N}$ is an orthonormal basis in
 $F_d$, and
$$W_d\,e_{j,d}=\lambda_{j,d}\,e_{j,d}.$$

From \cite[p. 118]{NW1}
 we get   that the $n$th minimal worst case
error is
$$e^{\rm wor}(n,d; \Lz^{\rm all})=(\lz_{n+1,d})^{1/2},$$
and  it is achieved by the optimal algorithm
$$S_{n,d}^*f=\sum_{k=1}^n \langle f, e_{k,d} \rangle_{F_d}\,
e_{k,d},$$that is,
\begin{equation}\label{2.2-0}e^{\rm wor}(n,d; \Lz^{\rm all})=\sup_{f\in F_d,\ \|f\|_{F_d}\le 1}\|f-S_{n,d}^*f\|_{G_d}=(\lz_{n+1,d})^{1/2}.\end{equation}

Without loss of generality, we may assume that all the eigenvalues
are positive. We set $$\eta_{k,d}=\lz_{k,d}^{-1/2}e_{k,d},\ \ k\in
\Bbb N.$$ We remark that $\{ e_{k,d}\}$ is an orthonormal basis in
$F_d$,  $\{ \eta_{k,d}\}$ is an orthonormal system in $G_d$,  and
for $f\in F_d$,
$$ \langle f, e_{k,d} \rangle_{{F_d}}\, e_{k,d}= \langle f, \eta_{k,d}
\rangle_{{G_d}}\, \eta_{k,d},$$and
\begin{equation}\label{2.2-1}S_{n,d}^*f=\sum_{k=1}^n \langle f,
\eta_{k,d} \rangle_{G_d}\, \eta_{k,d}\end{equation}

\subsection{Randomized setting}\

In the randomized setting, we consider randomized algorithms
$A_{n,d}^\oz$ of the form
\begin{equation}\label{2.3}A_{n,d}^\oz(f)=\phi_{n,d,\oz}(L_{1,\omega}(f),\cdots,L_{n,\omega}(f)),\
L_{j,\oz}\in\Lambda,\ 1\leq j\leq n,\end{equation}
 where $\Lz\in \{\Lambda^{\rm
all},\Lambda^{\rm std}\}$,    $\phi_{n,d,\oz}$ and $L_{j,\oz}$
could be randomly selected according to some probability space
$(\Omega, \Sigma, \mathcal P)$, for any fixed $\oz\in \Omega$,
$A_n^\oz$ is a deterministic method with cardinality $n=n(f,\oz)$,
the number $n=n(f,\oz)$ may be randomized and adaptively depend on
the input, and the cardinality of $(A_{n,d}^\oz)$ is then  defined
by
$${\rm Card}(A_{n,d}^\oz)=\sup_{f\in F_d,\ \|f\|_{F_d}\le 1}\mathbb E_\oz n(f,\oz).$$
The randomized approximation  error for the algorithm
$A_{n,d}^\oz$  of the form \eqref{2.3} is defined as
$$e^{\rm ran}(A_{n,d}^\oz)=\sup_{f\in F_d,\ \|f\|_{F_d}\leq 1 }\Big(\mathbb{E}_\oz\big\|{\rm APP}_d\,f-A_{n,d}^\oz(f)\big\|^{2}_{G_d}\Big)^{1/2}.$$
The $n$th minimal randomized  error for $\Lz\in\{\Lz^{\rm all},
\Lz^{\rm std}\}$ is defined by
$$e^{\rm ran}(n,d;\Lz)=\inf_{A_{n,d}^\oz \ {\rm with}\ L_{i,\oz}\in \Lz}e^{\rm ran}(A_{n,d}^\oz),$$
where the infimum is taken over all randomized algorithms
$A_{n,d}^\oz$ of the form \eqref{2.3} with ${\rm
Card}(A_{n,d}^\oz)\le  n$.

For $n=0$, we use $A_{0,d}^\oz=0$. We have
$$e^{\rm ran}(0,d;\Lz )=e^{\rm wor}(0,d;\Lz^{\rm all} )=(\lz_{1,d})^{1/2}.$$

There are many papers devoted to studying randomized approximation
and  relations of $e^{\rm ran}(n,d;\Lz)$ and $e^{\rm
wor}(n,d;\Lz)$
 for $\Lz\in\{\Lz^{\rm all}, \Lz^{\rm
std}\}$ (see \cite{BKN, FD1, FD2,  H92, K19, K18, KWW, M91, N88,
N92, NW1, NW3, TWW, WW07}).

This paper is devoted to discussing the equivalence of
tractability for $\Lz^{\rm all}$ and $\Lz^{\rm std}$ in the
randomized settings.  For $\Lz^{\rm std}$ the authors in
\cite{WW07, NW3, K19} used simplified randomized algorithms of the
form
\begin{equation}\label{2.4}A_{n,\vec t}(f)=\sum_{j=1}^n
f(t_j)g_{j,\vec t},\end{equation} where $\vec t=[t_1,\dots,t_n]$
for some random points $t_1,\dots, t_n$ from $D_d$, which are
independent, and each $t_j$ is distributed according to some
probability. The functions $g_{j,\vec t}\in G_d$ may depend on the
selected points $t_j$'s but are independent of $f$. For any $f$,
we view $A_{n,\cdot}(f)$ as a random process, and $A_{n,\vec
t}(f)$ as its specific realization.

We stress that algorithms of the form \eqref{2.4} belong to a
restricted class of all randomized algorithms, which are called
randomized linear algorithms. Indeed, we assume that $n$ is not
randomized, and for a fixed $\vec t$ we consider only linear
algorithms in $f(t_j)$. In this paper we also consider algorithms
of the form \eqref{2.4}. However, in \cite{WW07, NW3, K19} $t_j$'s
are assumed   to be independent, while in this paper we only
assume that $\vec t$ is distributed according to some probability,
and do not assume that $t_j$'s are independent.

The  information complexity  can be studied using either the
absolute error criterion (ABS) or the normalized error criterion
(NOR). For $\diamond\in\{{\rm wor,ran}\}$, $\star\in\{{\rm
ABS,\,NOR}\}$, and $\Lz\in \{\Lz^{\rm all}, \Lz^{\rm std}\}$, we
define the  information complexity $n^{\diamond, \star}(\va
,d;\Lz)$  as
\begin{equation} n^{\diamond, \star}(\va
,d;\Lz)=\inf\{n \ \big|\  e^{\diamond}(n,d;\Lz)\le \vz {\rm
CRI}_d\}, \end{equation} where
\begin{equation*}
{\rm CRI}_d=\left\{\begin{split}
 &\ \ 1, \; \quad\qquad\text{for $\star$=ABS,} \\
 &e^{\diamond}(0,d,\Lz), \text{ for $\star$=NOR}
\end{split}\right. \ \ =\ \ \left\{\begin{split}
 &\ 1, \; \qquad\ \ \ \qquad\text{ for $\star$=ABS,} \\
 &\big(\lz_{1,j}\big)^{1/2},\ \qquad\text{ for $\star$=NOR.}
\end{split}\right.
\end{equation*}
We remark that $$e^{\rm wor}(0,d,\Lz^{\rm all}) =e^{\rm
ran}(0,d,\Lz^{\rm all})=e^{\rm ran}(0,d,\Lz^{\rm std}).$$ Since
$\Lz^{\rm std}\subset \Lz^{\rm all},$ we get
\begin{equation*}e^{\rm ran}(n,d;\Lz^{\rm all})\le e^{\rm ran}(n,d;\Lz^{\rm std}).\end{equation*}
It follows that for $\star\in\{{\rm ABS,\,NOR}\}$,
\begin{equation} \label{2.4-1}n^{{\rm ran}, \star}(\va
,d;\Lz^{\rm all})\le n^{{\rm ran}, \star}(\va ,d;\Lz^{\rm std}).
\end{equation}

\subsection{Notions of tractability}\

In this subsection we briefly recall the various tractability
notions. Let ${\rm APP}= \{{\rm APP}_d\}_{d\in\Bbb N}$,
$\diamond\in\{{\rm wor,ran}\}$, $\star\in\{{\rm ABS,\,NOR}\}$, and
$\Lz\in \{\Lz^{\rm all}, \Lz^{\rm std}\}$. In the $\diamond$
setting for the class $\Lambda$, and for error criterion $\star$,
we say that ${\rm APP}$ is

$\bullet$ Algebraic strongly polynomially tractable (ALG-SPT) if
there exist $ C>0$ and non-negative number $p$ such that
\begin{equation}\label{2.7}n^{\diamond, \star}(\va ,d;\Lz)\leq C\varepsilon^{-p},\
\text{for all}\ \varepsilon\in(0,1).\end{equation}The exponent
ALG-$p^{\diamond, \star}(\Lz)$ of ALG-SPT is defined as the
infimum of $p$ for which \eqref{2.7} holds;

$\bullet$ Algebraic polynomially tractable (ALG-PT)  if there
exist $ C>0$ and non-negative numbers $p,q$ such that
$$n^{\diamond, \star}(\va
,d;\Lz)\leq Cd^{q}\varepsilon^{-p},\ \text{for all}\
d\in\mathbb{N},\ \varepsilon\in(0,1);$$

$\bullet$ Algebraic quasi-polynomially tractable (ALG-QPT) if
there exist $ C>0$ and non-negative number $t$ such that
\begin{equation}\label{2.8}n^{\diamond, \star}(\va
,d;\Lz)\leq C \exp(t(1+\ln{d})(1+\ln{\varepsilon^{-1}})),\
\text{for all}\ d\in\mathbb{N},\
\varepsilon\in(0,1).\end{equation}The exponent ALG-$t^{\diamond,
\star}(\Lz)$ of ALG-QPT  is defined as the infimum of $t$ for
which \eqref{2.8} holds;

$\bullet$ Algebraic uniformly weakly tractable (ALG-UWT)  if
$$\lim_{\varepsilon^{-1}+d\rightarrow\infty}\frac{\ln n^{\diamond, \star}(\va
,d;\Lz)}{\varepsilon^{-\alpha}+d^{\beta}}=0,\ \text{for all}\
\alpha, \beta>0;$$

$\bullet$ Algebraic weakly tractable (ALG-WT) if
$$\lim_{\varepsilon^{-1}+d\rightarrow\infty}\frac{\ln n^{\diamond, \star}(\va
,d;\Lz)}{\varepsilon^{-1}+d}=0;$$

$\bullet$ Algebraic $(s,t)$-weakly tractable (ALG-$(s,t)$-WT) for
fixed $s, t>0$ if
 $$\lim_{\varepsilon^{-1}+d\rightarrow\infty}\frac{\ln n^{\diamond, \star}(\va
,d;\Lz)}{\varepsilon^{-s}+d^{t}}=0.$$

Clearly, ALG-$(1,1)$-WT is the same as ALG-WT. If ${\rm APP}$ is
not ALG-WT, then ${\rm APP}$  is called  intractable.

If
 the $n$th  minimal error    decays faster
than any polynomial and is exponentially convergent, then we
should study  tractability with $\vz^{-1}$ being replaced by
$(1+\ln \frac 1{\vz})$, which is called exponential tractability.
Recently, there have been many papers studying exponential
tractability (see \cite{DLPW, DKPW, X3, PPXD, KW, IKPW, CW, LX2}).

In the definitions of ALG-SPT, ALG-PT, ALG-QPT, ALG-UWT, ALG-WT,
and ALG-$(s,t)$-WT, if we replace $\frac1{\vz}$ by $(1+\ln \frac
1{\vz})$, we get the definitions of \emph{exponential strong
polynomial tractability} (EXP-SPT), \emph{exponential polynomial
tractability} (EXP-PT), \emph{exponential quasi-polynomial
tractability} (EXP-QPT), \emph{exponential uniform weak
tractability}
 (EXP-UWT), \emph{exponential weak tractability}
 (EXP-WT), and \emph{exponential $(s,t)$-weak tractability}
 (EXP-$(s,t)$-WT), respectively.  We now give the above notions of exponential tractability in
 detail.

Let ${\rm APP}= \{{\rm APP}_d\}_{d\in\Bbb N}$, $\diamond\in\{{\rm
wor,ran}\}$, $\star\in\{{\rm ABS,\,NOR}\}$, and $\Lz\in \{\Lz^{\rm
all}, \Lz^{\rm std}\}$. In the $\diamond$ setting for the class
$\Lambda$, and for error criterion $\star$, we say that ${\rm
APP}$ is

$\bullet$ Exponential strongly polynomially tractable (EXP-SPT) if
there exist $ C>0$ and non-negative number $p$ such that
\begin{equation}\label{2.9}n^{\diamond, \star}(\va ,d;\Lz)\leq C(\ln\varepsilon^{-1}+1)^{p},\
\text{for all}\ \varepsilon\in(0,1).\end{equation}The exponent
EXP-$p^{\diamond, \star}(\Lz)$ of EXP-SPT is defined as the
infimum of $p$ for which \eqref{2.9} holds;

$\bullet$ Exponential  polynomially tractable (EXP-PT)  if there
exist $C>0$ and non-negative numbers $p,q$ such that
$$n^{\diamond, \star}(\va
,d;\Lz)\leq Cd^{q}(\ln\varepsilon^{-1}+1)^{p},\ \text{for all}\
d\in\mathbb{N},\ \varepsilon\in(0,1);$$

$\bullet$ Exponential  quasi-polynomially tractable (EXP-QPT) if
there exist $C>0$ and non-negative number $t$ such that
\begin{equation}\label{2.10}n^{\diamond, \star}(\va
,d;\Lz)\leq C \exp(t(1+\ln{d})(1+\ln(\ln\varepsilon^{-1}+1))),\
\text{for all}\ d\in\mathbb{N},\
\varepsilon\in(0,1).\end{equation}The exponent EXP-$t^{\diamond,
\star}(\Lz)$ of EXP-QPT  is defined as the infimum of $t$ for
which \eqref{2.10} holds;

$\bullet$ Exponential  uniformly weakly tractable (EXP-UWT)  if
$$\lim_{\varepsilon^{-1}+d\rightarrow\infty}\frac{\ln n^{\diamond, \star}(\va
,d;\Lz)}{(1+\ln\varepsilon^{-1})^{\alpha}+d^{\beta}}=0,\ \text{for
all}\ \alpha, \beta>0;$$

$\bullet$ Exponential  weakly tractable (EXP-WT) if
$$\lim_{\varepsilon^{-1}+d\rightarrow\infty}\frac{\ln n^{\diamond, \star}(\va
,d;\Lz)}{1+\ln\varepsilon^{-1}+d}=0;$$

$\bullet$ Exponential  $(s,t)$-weakly tractable (EXP-$(s,t)$-WT)
for fixed $s,t>0$ if
 $$\lim_{\varepsilon^{-1}+d\rightarrow\infty}\frac{\ln n^{\diamond, \star}(\va
,d;\Lz)}{(1+\ln\varepsilon^{-1})^{s}+d^{t}}=0.$$

\subsection{Equivalences of tractability for $\Lz^{\rm all}$ in
the worst case and randomized settings}\

In this subsection, we introduce the equivalences of tractability
for $\Lz^{\rm all}$ in the worst case and randomized settings. It
follows from \cite{N92}, \cite[p. 284]{NW1} that
$$\frac12 e^{\rm wor}(4n-1,d;\Lz^{\rm all})\le e^{\rm ran}(n,d;\Lz^{\rm all})\le e^{\rm wor}(n,d;\Lz^{\rm
all}),
$$which means that for $\star \in \{{\rm ABS, NOR}\}$ and  $n^{\rm
ran,\star}(\vz,d;\Lz^{\rm all})\ge1$,
\begin{equation}\label{2.014}\frac{1}{4}\left(n^{\rm wor,\star}(2\varepsilon,d;\Lambda^{\rm all})+1\right)\leq n^{\rm ran,\star}(\varepsilon,d;\Lambda^{\rm all})
\leq n^{\rm wor,\star}(\varepsilon,d;\Lambda^{\rm
all}).\end{equation}

From \eqref{2.014} we  get  the equivalences of tractability for
$\Lz^{\rm all}$ in the worst case and randomized settings. Indeed,
for the absolute or normalized error criterion,  \cite[Corollaries
22.1 and 22.2]{NW3} shows  the equivalences of ALG-SPT, ALG-PT,
ALG-QPT, ALG-WT  for  $\Lz^{\rm all}$   in the worst case and
randomized settings, and that  the exponents of ALG-SPT and
ALG-QPT in the worst case and randomized settings are also same.

Using \eqref{2.014} and the same method as in the proof of
\cite[Corollaries 22.1 and 22.2]{NW3}, for the absolute or
normalized error criterion we obtain  the equivalences of ALG-UWT,
ALG-$(s,t)$-WT, EXP-SPT, EXP-PT, EXP-QPT, EXP-WT, EXP-UWT,
EXP-$(s,t)$-WT, for $\Lz^{\rm all}$   in the worst case and
randomized settings, and that  the exponents of EXP-SPT and
EXP-QPT in the worst case and randomized settings are also same.

We remark that in showing  EXP-$t^{\rm ran,\star}(\Lz^{\rm all})=$
EXP-$t^{\rm wor,\star}(\Lz^{\rm all})$, we use the following
inequalities (see \cite[p. 43]{NW3}):  for $\dz\in(0,1)$ and
$n^{\rm ran,\star}(\vz,d;\Lz^{\rm all})\ge1$,
\begin{equation*}
\dz^2\Big(n^{\rm
wor,\star}(\frac{\varepsilon}{1-\dz},d;\Lambda^{\rm
all})+1\Big)\leq n^{\rm ran,\star}(\varepsilon,d;\Lambda^{\rm
all}) \leq n^{\rm wor,\star}(\varepsilon,d;\Lambda^{\rm all}),
\end{equation*}instead of \eqref{2.014}. See the  proof of
Theorem 5.4.

We summarize these properties in the next corollary.

\begin{cor}Consider the approximation problem ${\rm APP}= \{{\rm APP}_d\}_{d\in\Bbb N}$ for the absolute
or normalized error criterion in the randomized and worst case
settings for  $\Lz^{\rm all}$. Then \vskip 2mm

$\bullet$ $\rm ALG$-$\rm SPT$, $\rm ALG$-$\rm PT$, $\rm ALG$-$\rm
QPT$,   $\rm ALG$-$\rm UWT$, $\rm ALG$-$\rm WT$,    $\rm
ALG$-$(s,t)$-$\rm WT$ in the randomized setting is equivalent to
$\rm ALG$-$\rm SPT$, $\rm ALG$-$\rm PT$, $\rm ALG$-$\rm QPT$, $\rm
ALG$-$\rm UWT$, $\rm ALG$-$\rm WT$,  $\rm ALG$-$(s,t)$-$\rm WT$ in
the worst case setting; \vskip 2mm

$\bullet$ $\rm EXP$-$\rm SPT$, $\rm EXP$-$\rm PT$, $\rm EXP$-$\rm
QPT$, $\rm EXP$-$\rm UWT$, $\rm EXP$-$\rm WT$, $\rm
EXP$-$(s,t)$-$\rm WT$ in the randomized setting is equivalent to
$\rm EXP$-$\rm SPT$, $\rm EXP$-$\rm PT$, $\rm EXP$-$\rm QPT$, $\rm
EXP$-$\rm UWT$, $\rm EXP$-$\rm WT$, $\rm EXP$-$(s,t)$-$\rm WT$,
$\rm ALG$-$\rm UWT$, $\rm ALG$-$(s,t)$-$\rm WT$ in the worst case
setting; \vskip 2mm

$\bullet$  the exponents of {\rm SPT} and {\rm QPT} are the same
in the two settings, i.e., for $\star \in \{{\rm ABS, NOR}\}$,
\begin{align*}{\rm ALG}\!-\!p^{\rm wor,\star}(\Lz^{\rm all}) &= {\rm ALG}\!-\!p^{\rm ran,\star}(\Lz^{\rm all}),\\
{\rm ALG}\!-\!t^{\rm wor,\star}(\Lz^{\rm all}) &= {\rm
ALG}\!-\!t^{\rm ran,\star}(\Lz^{\rm all}), \\
{\rm EXP}\!-\!p^{\rm wor,\star}(\Lz^{\rm all}) &= {\rm EXP}\!-\!p^{\rm ran,\star}(\Lz^{\rm all}),\\
{\rm EXP}\!-\!t^{\rm wor,\star}(\Lz^{\rm all}) &= {\rm
EXP}\!-\!t^{\rm ran,\star}(\Lz^{\rm all}). \end{align*}

\end{cor}

\subsection{Main results}\

We shall give  main results of this paper in this subsection. The
first important progress about the relation between $e^{\rm
ran}(n,d;\Lz^{\rm std})$ and $e^{\rm wor}(n,d;\Lz^{\rm all})$ was
obtained by Wasilkowski and Wo\'zniakowski in \cite{WW07} by
constructing iterated  Monte Carlo methods. They showed that the
powers of $e^{\rm ran}(n,d;\Lz^{\rm std})$ and $e^{\rm
wor}(n,d;\Lz^{\rm all})$ are same, and obtained the equivalences
of ALG-SPT and ALG-PT for  $\Lz^{\rm all}$ and $\Lz^{\rm std}$ for
the normalized error criterion in the randomized setting. Novak
and Wo\'zniakowski in \cite{NW3} and Krieg in \cite{K19} refined
the above randomized algorithms and showed that $e^{\rm
ran}(n,d;\Lz^{\rm std})$ is asymptotically of the same order as
$e^{\rm wor}(n,d;\Lz^{\rm all})$ given that $e^{\rm
wor}(n,d;\Lz^{\rm all})$ is regularly decreasing. However, the
obtained  relations are heavily dependent of the initial error,
and are not sharp if $e^{\rm wor}(n,d;\Lz^{\rm all})$ is
exponentially convergent.

If  nodes ${\rm X} = (x^1,\dots, x^n)\in D_d^n$ are drawn
independently and identically distributed according to a
probability measure, then the  samples on the nodes ${\rm X}$ is
called the random information (see \cite{HKNPU, HKNPU1, KS}).
Using  random information and the  least squares method we can
obtain the relation between $e^{\rm ran}(n,d;\Lz^{\rm std})$ and
$e^{\rm wor}(n,d;\Lz^{\rm all})$ (see \cite{CDL, CM}). The authors
in \cite{KUV} used  random information satisfying some condition
and the least squares method to obtain an inequality  between
$e^{\rm ran}(n,d;\Lz^{\rm std})$ and $e^{\rm wor}(n,d;\Lz^{\rm
all})$  (see \cite[Theorem 6.1]{KUV}). They remarked in
\cite[Remark 6.3]{KUV} that using the weighed least squares method
can improve the above inequality.

In this paper we  use the method proposed in \cite[Remark
6.3]{KUV}, i.e.,  combining the proof of   \cite[Theorem 6.1]{KUV}
with the weighed least squares method used in \cite{CM}, to get an
improved inequality between $e^{\rm ran}(n,d;\Lz^{\rm std})$ and
$e^{\rm wor}(n,d;\Lz^{\rm all})$. See the following theorem.
Compared with the results in \cite{NW3, K19}, our inequality does
not depend on the initial error, and  are almost sharp if $e^{\rm
wor}(n,d;\Lz^{\rm all})$ is exponentially convergent. However, if
$e^{\rm wor}(n,d;\Lz^{\rm all})$ is regularly decreasing, then by
our inequality we can only obtain that $e^{\rm ran}(n,d;\Lz^{\rm
std})$ is  at most asymptotically of the  order of $e^{\rm
wor}(m,d;\Lz^{\rm all})$, where $n$ is at least of order $m \ln
m$.

\begin{thm}
Let $\delta\in(0,1)$,  $m,n\in\mathbb{N}$ be such that
$$m=\left\lfloor\frac{n}{48(\sqrt{2}\ln(2n)-\ln{\delta})}\right\rfloor.$$
Then we have
\begin{equation}\label{2.12}e^{\rm ran}(n,d;\Lz^{\rm std})\le\Big(1+\frac{4m}{n}\Big)^{\frac12}\frac{1}{\sqrt{1-\delta}}\,e^{\rm wor}(m,d;\Lz^{\rm all}),
\end{equation}where $\lfloor x\rfloor$ denotes the largest
integer not exceeding $x$.
\end{thm}

Based on Theorem 2.2, we obtain two relations between the
information complexities $n^{\rm
ran,\star}(\varepsilon,d;\Lambda^{\rm std})$ and $n^{\rm
wor,\star}(\varepsilon,d;\Lambda^{\rm all})$ for $\star\in\{{\rm
ABS,\,NOR}\}$.
\begin{thm}For $\star\in\{{\rm ABS,\,NOR}\}$, we have \begin{align}\label{2.14}n^{\rm
ran,\star}(\varepsilon,d;\Lambda^{\rm std})\le  96\sqrt2
\Big(n^{\rm wor,\star}(\frac\varepsilon4,d;\Lambda^{\rm
all})+1\Big) \Big(\ln\big(n^{\rm
wor,\star}(\frac\varepsilon4,d;\Lambda^{\rm
all})+1\big)+\ln(192\sqrt{2})\Big). \end{align}Furthermore, for
sufficiently small $\delta>0$, we have
\begin{align}\label{2.15}
n^{\rm ran,\star}(\varepsilon,d;\Lambda^{\rm std}) \le
48\Big(4\big(\ln48&+\ln\ln{\frac{1}{\delta}}+ \ln\big(n^{\rm
wor,\star}(\frac{\varepsilon}{A_\delta},d;\Lambda^{\rm
all})+1\big)\big)\\ &+\ln{\frac{1}{\delta}}\Big) \big(n^{\rm
wor,\star}(\frac{\varepsilon}{A_\delta},d;\Lambda^{\rm
all})+1\big),\nonumber
\end{align}where $A_{\delta}:=\Big(1+\frac{1}{12\ln{\frac{1}{\delta}}}\Big)^{\frac12}\frac{1}{\sqrt{1-\delta}}$.
\end{thm}

It is easy to see that  for any $\oz,\dz>0$,
\begin{equation}\label{2.16}\sup_{x\ge1}\frac{96\sqrt2(\ln x +\ln(192\sqrt{2}))}{x^{\oz}}= C_\oz <+\infty. \end{equation}
and
\begin{equation}\label{2.16-0}\sup_{x\ge1}\frac{48(4(\ln48+\ln\ln{\frac{1}{\delta}}+ \ln
x)+\ln{\frac{1}{\delta}})}{x^{\oz}}= C_{\oz,\dz} <+\infty.
\end{equation}

According to \eqref{2.14}-\eqref{2.16-0}, we have the following
corollary which gives two  useful inequalities between the
information complexities $n^{\rm
ran,\star}(\varepsilon,d;\Lambda^{\rm std})$ and $n^{\rm
wor,\star}(\varepsilon,d;\Lambda^{\rm all})$ for $\star\in\{{\rm
ABS,\,NOR}\}$.

\begin{cor} For  $\star\in\{{\rm ABS,\,NOR}\}$ and $\oz>0$, we
have
\begin{equation}\label{2.17}n^{\rm
ran,\star}(\varepsilon,d;\Lambda^{\rm std})\le C_\oz \Big(n^{\rm
wor,\star}(\frac\varepsilon4,d;\Lambda^{\rm all})+1\Big)^{1+\oz}.
\end{equation}
Similarly,  for sufficiently small $\oz,\dz>0$ and $\star\in\{{\rm
ABS,\,NOR}\}$, we have
\begin{equation}\label{2.18}
n^{\rm ran,\star}(\varepsilon,d;\Lambda^{\rm std}) \le C_{\oz,\dz}
 \big(n^{\rm
wor,\star}(\frac{\varepsilon}{A_\delta},d;\Lambda^{\rm
all})+1\big)^{1+\oz},
\end{equation}where $A_{\delta}:=\Big(1+\frac{1}{12\ln{\frac{1}{\delta}}}\Big)^{\frac12}\frac{1}{\sqrt{1-\delta}}$.

\end{cor}

In the randomized setting,  for the normalized error criterion,
\cite[Theorems 22.19, 22.21, and 22.5]{NW3} gives the equivalences
of  ALG-PT (ALG-SPT), ALG-QPT, ALG-WT for $\Lz^{\rm all}$ and
$\Lz^{\rm std}$, and  shows that the exponents of ALG-SPT and
ALG-QPT for $\Lz^{\rm all}$ and $\Lz^{\rm std}$  are same. For the
absolute error criterion, \cite[Theorems 22.20, 22.22, and
22.6]{NW3} gives the equivalences of  ALG-PT (ALG-SPT), ALG-QPT,
ALG-WT for $\Lz^{\rm all}$ and $\Lz^{\rm std}$ under some
conditions on the initial error $\sqrt{\lz_{1,d}}$. Novak and
Wo\'zniakowski posed Open problems 98, 101, 102 in \cite{NW3}
which ask whether  the above conditions are necessary.

In this paper  we obtain the equivalences of ALG-SPT, ALG-PT,
ALG-QPT, ALG-WT  for $\Lz^{\rm all}$ and $\Lz^{\rm std}$ for the
absolute error criterion without any condition, which means the
above conditions  are unnecessary. This solves Open problems 98,
101, 102 in \cite{NW3}. See the following theorem.

\begin{thm} Consider the problem ${\rm APP}=\{{\rm APP}_d\}_{d\in \Bbb N}$
in the randomized  setting for the absolute error criterion.
Then\vskip 2mm

$\bullet$ $\rm ALG$-$\rm SPT$, $\rm ALG$-$\rm PT$, $\rm ALG$-$\rm
QPT$, $\rm ALG$-$\rm WT$ for $\Lz^{\rm all}$ is equivalent to $\rm
ALG$-$\rm SPT$, $\rm ALG$-$\rm PT$, $\rm ALG$-$\rm QPT$,  $\rm
ALG$-$\rm WT$ for $\Lz^{\rm std}$;\vskip 2mm

$\bullet$ The exponents $\rm ALG$-$p^{\rm ran, ABS}(\Lz)$ of $\rm
ALG$-$\rm SPT$ for $\Lz^{\rm all}$ and $\Lz^{\rm std}$ are same,
and the exponents $\rm ALG$-$t^{\rm ran, ABS}(\Lz)$ of $\rm
ALG$-$\rm QPT$ for $\Lz^{\rm all}$ and $\Lz^{\rm std}$ are also
same.
\end{thm}

In the randomized setting for the normalized or absolute error
criterion, the equivalences of ALG-UWT and ALG-$(s,t)$-WT, and the
various notions of  EXP-tractability for $\Lz^{\rm all}$ and
$\Lz^{\rm std}$,  as far as we know,  have not been studied. In
this paper, we investigate the problem and obtain the following
theorem which gives the above equivalences without any condition.

\begin{thm}
Consider the problem ${\rm APP}=\{{\rm APP}_d\}_{d\in \Bbb N}$ in
the randomized  setting  for the absolute  or  normalized error
criterion. Then for $\star\in\{{\rm ABS,NOR}\}$,\vskip 2mm

$\bullet$ $\rm EXP$-$\rm SPT$, $\rm EXP$-$\rm PT$, $\rm EXP$-$\rm
QPT$, $\rm EXP$-$\rm UWT$, $\rm EXP$-$\rm WT$, $\rm
EXP$-$(s,t)$-$\rm WT$, $\rm ALG$-$\rm UWT$,  $\rm
ALG$-$(s,t)$-$\rm WT$ for $\Lz^{\rm all}$ is equivalent to $\rm
EXP$-$\rm SPT$, $\rm EXP$-$\rm PT$, $\rm EXP$-$\rm QPT$, $\rm
EXP$-$\rm UWT$, $\rm EXP$-$\rm WT$, $\rm EXP$-$(s,t)$-$\rm WT$,
$\rm ALG$-$\rm UWT$,  $\rm ALG$-$(s,t)$-$\rm WT$ for $\Lz^{\rm
std}$;\vskip 2mm

$\bullet$ The exponents $\rm EXP$-$p^{\rm ran, \star}(\Lz)$ of
$\rm EXP$-$\rm SPT$ for $\Lz^{\rm all}$ and $\Lz^{\rm std}$ are
same, and the exponents $\rm EXP$-$t^{\rm ran, \star}(\Lz)$ of
$\rm EXP$-$\rm QPT$ for $\Lz^{\rm all}$ and $\Lz^{\rm std}$ are
also same.

\end{thm}

Combining Corollary 2.1 with Theorems 2.5 and 2.6 we obtain the
following corollary.

\begin{cor}Consider the approximation problem ${\rm APP}= \{{\rm APP}_d\}_{d\in\Bbb N}$ for the absolute
or normalized error criterion in the randomized and worst case
settings. Then \vskip 2mm

$\bullet$ $\rm ALG$-$\rm SPT$, $\rm ALG$-$\rm PT$, $\rm ALG$-$\rm
QPT$, $\rm ALG$-$\rm UWT$,   $\rm ALG$-$\rm WT$,  $\rm
ALG$-$(s,t)$-$\rm WT$ in the worst case setting for   $\Lz^{\rm
all}$ is equivalent to $\rm ALG$-$\rm SPT$, $\rm ALG$-$\rm PT$,
$\rm ALG$-$\rm QPT$, $\rm ALG$-$\rm UWT$, $\rm ALG$-$\rm WT$, $\rm
ALG$-$(s,t)$-$\rm WT$ in the randomized setting for $\Lz^{\rm
all}$  or for $\Lz^{\rm std}$; \vskip 2mm

$\bullet$ $\rm EXP$-$\rm SPT$, $\rm EXP$-$\rm PT$, $\rm EXP$-$\rm
QPT$, $\rm EXP$-$\rm UWT$, $\rm EXP$-$\rm WT$, $\rm
EXP$-$(s,t)$-$\rm WT$ in the worst case setting for   $\Lz^{\rm
all}$ is equivalent to $\rm EXP$-$\rm SPT$, $\rm EXP$-$\rm PT$,
$\rm EXP$-$\rm QPT$, $\rm EXP$-$\rm UWT$, $\rm EXP$-$\rm WT$, $\rm
EXP$-$(s,t)$-$\rm WT$ in the randomized setting for  $\Lz^{\rm
all}$  or for $\Lz^{\rm std}$;  \vskip 2mm

$\bullet$  the exponents of ${\rm SPT}$ and ${\rm QPT}$ are the
same in the worst case setting for   $\Lz^{\rm all}$ and in the
randomized setting for  $\Lz^{\rm all}$  and  $\Lz^{\rm std}$,
i.e., for $\star \in \{{\rm ABS, NOR}\}$,
\begin{align*}{\rm ALG}\!-\!p^{\rm wor,\star}(\Lz^{\rm all}) &= {\rm ALG}\!-\!p^{\rm ran,\star}(\Lz^{\rm all})= {\rm ALG}\!-\!p^{\rm ran,\star}(\Lz^{\rm std}),\\
{\rm ALG}\!-\!t^{\rm wor,\star}(\Lz^{\rm all}) &= {\rm
ALG}\!-\!t^{\rm ran,\star}(\Lz^{\rm all})= {\rm
ALG}\!-\!t^{\rm ran,\star}(\Lz^{\rm std}), \\
{\rm EXP}\!-\!p^{\rm wor,\star}(\Lz^{\rm all}) &= {\rm EXP}\!-\!p^{\rm ran,\star}(\Lz^{\rm all})= {\rm EXP}\!-\!p^{\rm ran,\star}(\Lz^{\rm std}),\\
{\rm EXP}\!-\!t^{\rm wor,\star}(\Lz^{\rm all}) &= {\rm
EXP}\!-\!t^{\rm ran,\star}(\Lz^{\rm all})= {\rm EXP}\!-\!t^{\rm
ran,\star}(\Lz^{\rm std}). \end{align*}

\end{cor}

\section{Proofs of Theorems 2.2 and 2.3}

Let us keep the notation of  Subsection 2.1. For any
$m\in\mathbb{N}$, we define the  functions $ h_{m,d}(x)$ and
$\oz_{m,d}$ on $D_d$ by
$$h_{m,d}(x):=\frac{1}{m}\sum_{j=1}^{m}|\eta_{j,d}(x)|^{2},\ \ \
\oz_{m,d}(x):=h_{m,d}(x)\,\rho_d(x),$$where
$\{\eta_{j,d}\}_{j=1}^\infty$ is an orthonormal system in
$G_d=L_2(D_d,\rho_d(x)dx)$.  Then $\oz_{m,d}$ is a probability
density function on $D_d$, i.e., $\int_{D_d}\oz_{m,d}(x)\,dx=1$.
We define the corresponding probability measure $\mu_{m,d}$ by
$$\mu_{m,d}(A)=\int_{A}\oz_{m,d}(x)\,dx,$$where $A$ is a Borel subset of
$D_d$. We use the convention that $\frac00:=0$. Then $\{\tilde
\eta_{j,d}\}_{j=1}^\infty$ is an orthonormal system in
$L_2(D_d,\mu_{m,d})$, where $$ \tilde \eta_{j,d}:=\frac
{\eta_{j,d}} {\sqrt{h_{m,d}}} .$$

For ${\rm X}=(x^1, \dots, x^n)\in D_d^n$, we use the following
matrices
\begin{equation}
\widetilde{L}_m=\widetilde{L}_m({\rm X})=\left(
\begin{array}{cccc}
\widetilde{\eta}_{1,d}(x^{1})&\widetilde{\eta}_{2,d}(x^{1})&\cdots&\widetilde{\eta}_{m,d}(x^{1})\\
\widetilde{\eta}_{1,d}(x^{2})&\widetilde{\eta}_{2,d}(x^{2})&\cdots&\widetilde{\eta}_{m,d}(x^{2})\\
\vdots&\vdots&\ &\vdots\\
\widetilde{\eta}_{1,d}(x^{n})&\widetilde{\eta}_{2,d}(x^{n})&\cdots&\widetilde{\eta}_{m,d}(x^{n})\\
\end{array}
\right)\ \ \ \ \ {\rm and} \ \ \ \ \ \widetilde H_m=\frac
1n\widetilde{L}_m^*\widetilde{L}_m,\label{3.1}
\end{equation}where $A^*$ is the conjugate transpose of a matrix $A$.
Note that $$\widetilde{N}(m):=\sup\limits_{x\in
D_d}\sum\limits_{k=1}^{m}|\widetilde{\eta}_{k,d}(x)|^{2}=m.$$

According to  \cite[Propositions 5.1 and 3.1]{KUV} we have the
following results.

\begin{lem}
Let $n,m\in\mathbb{N}$. Let $x^{1},\ldots,x^{n}\in D_d$ be drawn
independently and identically distributed at random with respect
to the probability measure $\mu_{m,d}$. Then it holds for $0<t<1$
that
$$\mathbb{P}(\|\widetilde{H}_{m}-I_{m}\|>t)\leq(2n)^{\sqrt{2}}\exp\left(-\frac{nt^{2}}{12m}\right),
$$where $\widetilde{L}_m, \ \widetilde H_m$ are given by \eqref{3.1},  $I_m$ is the identity matrix of order
$m$, and $\|L\|$ denotes the spectral norm (i.e. the largest
singular value) of a matrix $L$.
\end{lem}

\begin{lem}
Let $n,m\in\mathbb{N}$, and let $\widetilde{L}_m, \ \widetilde
H_m$ be given by \eqref{3.1}.  If
$$\|\widetilde{H}_{m}-I_{m}\|\le 1/2,$$then
\begin{equation}\|(\widetilde{L}^{*}_m\widetilde{L}_m)^{-1}\|\le \frac2n.\label{3.3}\end{equation}
\end{lem}

\begin{rem}From Lemma 3.1 we immediately obtain that the matrix $\widetilde H_m \in \Bbb C^{m\times m}$ has only eigenvalues larger than $t := 1/2$
and satisfies $$\|\widetilde{H}_{m}-I_{m}\|\le 1/2$$ with
probability at least $1-\delta$ if
\begin{equation*}
\widetilde{N}(m)=m\le \frac{n}{48(\sqrt{2}\ln(2n)-\ln{\delta})}.
\end{equation*}

Specifically, if  \begin{equation} m=
\Big\lfloor\frac{n}{48(\sqrt{2}\ln(2n)-\ln{\delta})}\Big\rfloor\ge
1,\label{3.2}
\end{equation}  then the matrix $\widetilde H_m $ has only eigenvalues larger than $1/2$ and satisfies $$\|\widetilde{H}_{m}-I_{m}\|\le 1/2$$ with probability
at least $1-\delta$, where $\lfloor x\rfloor$ denotes the largest
integer not exceeding $x$. It follows that
\begin{equation}\label{3.4}\mathbb{P}\big(\|\widetilde{H}_{m}-I_{m}\|\le 1/2\big)\ge 1-\delta\end{equation}
holds given the condition \eqref{3.2}.
\end{rem}

Now let $m, n\in\Bbb N$ satisfying \eqref{3.2}, $x^1,\dots,x^n$ be
independent and identically distributed sample points  from $D_d$
that are distributed according to the probability measure
$\mu_{m,d}$, and  $\widetilde L_m, \widetilde H_m$ be given by
\eqref{3.1}. If the sample points ${\rm X}=(x^1,\dots,x^n)$
satisfy $\|\widetilde{H}_{m}-I_{m}\|> 1/2$, then we discard these
points and re-sample until the the re-sample points satisfy
$\|\widetilde{H}_{m}-I_{m}\|\le 1/2$. That is, we consider the
conditional distribution given the event
$\|\widetilde{H}_{m}-I_{m}\|\le 1/2$ and the conditional
expectation
$$\Bbb E(X\,\big|\ \|\widetilde{H}_{m}-I_{m}\|\le 1/2)=\frac{\int_{\|\widetilde{H}_{m}-I_{m}\|\le
1/2} X(x^1,\dots,x^n)\,d\mu_{m,d}(x^1)\dots d\mu_{m,d}(x^n)}
{\mathbb{P}\big(\|\widetilde{H}_{m}-I_{m}\|\le 1/2\big)}$$ of a
random variable $X$.

If $\|\widetilde{H}_{m}-I_{m}\|\le 1/2$ for some ${\rm
X}=(x^1,\dots,x^n)\in D_d^n$, then
$\widetilde{L}_{m}=\widetilde{L}_{m}({\rm X}) $ has the full rank.
The algorithm is a weighted least squares estimator
\begin{equation}{S}^{m}_{\rm X}\,f=\mathop{\arg\min}_{g\in
V_m}
\frac{|f(x^i)-g(x^i)|^{2}}{h_{m,d}(x^i)},\label{3.5}\end{equation}
where $V_m:={\rm span}\{\eta_{1,d},\dots,\eta_{m,d}\}$. It follows
that ${S}^{m}_{\rm X}\,f=f$ whenever $f\in V_m$.
\\

\begin{table}[!ht]
\begin{tabular}{clllrrrr}
\toprule
\multicolumn{6}{l}{\textbf{Algorithm} \quad Weighted least squares regression.}
\\
\midrule
Input:&${\rm X}=(x^1,\dots,x^n)\in D_d^n$  &    &set of distinct sampling nodes,\\
&$\tilde {\rm
f}=\Big(\frac{f(x^1)}{\sqrt{h_{m,d}(x^1)}},\dots,\frac{f(x^n)}{\sqrt{h_{m,d}(x^n)}}\Big)^T$
&
&weighted samples of $f$ evaluted\\
& & &at the nodes from ${\rm X}$,\\
&$m\in \Bbb N$  &    &$m<n$ such that the matrix \\
& & & $\widetilde{L}_{m}:=\widetilde{L}_{m}({\rm X})$ from \eqref{3.1} has\\
& & &full (column) rank.\\
\multicolumn{6}{l}{Solve the over-determined linear system}\\

\multicolumn{6}{c}{$\widetilde{L}_{m}(\widetilde
c_{1},\cdots,\widetilde c_{m})^{T}=\tilde {\rm
f}$}\\

\multicolumn{6}{l}{via least square, i.e., compute}\\

\multicolumn{6}{c}{$(\widetilde c_{1},\cdots,\widetilde
c_{m})^{T}=(\widetilde{L}^{*}_{m}\widetilde{L}_{m})^{-1}\widetilde{L}_{m}^{*}\
\tilde {\rm
f}$.}\\

\multicolumn{6}{l}{Output: $\widetilde c=(\widetilde c_{1},\cdots,\widetilde c_{m})^{T}\in\Bbb C^m$ coefficients of
 the approximant $ S_{\rm X}^{m}(f):=\sum\limits_{k=1}^{m}\widetilde c_{k}\eta_{k,d}$ }\\

 \multicolumn{6}{r} {which is the unique solution of \eqref{3.5}.} \\

\bottomrule
\end{tabular}
\end{table}

\noindent{\it Proof of Theorem 2.2. }

We have
$$(e^{\rm ran}(n,d;\Lz^{\rm std}))^2\le \Bbb E(\|f- S_{\rm X}^{m}(f)\|_{G_d}^2 \ \big|\ \|\widetilde{H}_{m}-I_{m}\|\le 1/2),$$
where $m,n\in\Bbb N$ satisfy \eqref{3.2}.  We estimate $\|f-
S_{\rm X}^{m}(f)\|_{G_d}^2$. We set $$H_d=L_2(D_d, \mu_{m,d}).$$
We recall that $\{e_{j,d}\}_{j=1}^\infty$ is an orthonormal basis
in $F_d$,  $\{\eta_{j,d}\}_{j=1}^\infty$ is an orthonormal system
in $G_d=L_2(D_d,\rho_d(x)dx)$, and  $\{\tilde
\eta_{j,d}\}_{j=1}^\infty$ is an orthonormal system in
$H_d=L_2(D_d,\mu_{m,d})$, where $$\eta_{j,d}=\lz_{j,d}^{-1/2}
e_{j,d},\ \ \  \tilde \eta_{j,d}:=\frac {\eta_{j,d}}
{\sqrt{h_{m,d}}} .$$

For $f\in F_d$ with $\|f\|_{F_d}\le 1$, we have
$$f=\sum_{k=1}^\infty \langle f,
e_{k,d} \rangle_{F_d}\, e_{k,d}=\sum_{k=1}^\infty \langle f,
\eta_{k,d} \rangle_{G_d}\, \eta_{k,d},$$and
$$\|f\|_{F_d}^2=\sum_{k=1}^\infty |\langle f, e_{k,d}
\rangle_{F_d}|^2= \sum_{k=1}^\infty \lz_{k,d}^{-1}|\langle f,
\eta_{k,d} \rangle_{G_d}|^2.$$

We note that $f-S_{m,d}^*(f)$ is orthogonal to the space $V_m$
with respect to the inner product
$\langle\cdot,\cdot\rangle_{G_d}$, and
$$S_{m,d}^*(f)-S^{m}_{\rm X}(f)=S^{m}_{\rm X}(f-S_{m,d}^*(f))\in V_m:={\rm span}\{\eta_{1,d},\dots, \eta_{m,d}\}, $$
where $$S_{m,d}^*(f)=\sum\limits_{k=1}^{m}\langle
f,\eta_{k,d}\rangle_{G_d}\eta_{k,d}.$$ It follows that
\begin{align*}\|f-S_{\rm
X}^{m}(f)\|_{G_d}^2&=\|f-S_{m,d}^*(f)\|_{G_d}^2+\|S_{\rm
X}^{m}(f-S_{m,d}^*(f))\|_{G_d}^2\\ &=\|g\|_{G_d}^2+\|S_{\rm
X}^{m}(g)\|_{G_d}^2,\end{align*}where $g=f-S_{m,d}^*(f)$.

We recall that
$$S_{\rm X}^m(g)=\sum_{k=1}^m\widetilde{c}_k\eta_{k,d},\,\widetilde{c}=(\widetilde{c}_1,\dots,\widetilde{c}_m)^T
=((\widetilde{L}_{m})^{*}\widetilde{L}_{m})^{-1}(\widetilde{L}_{m})^{*}\widetilde{\rm
g},$$where
$$\widetilde{\rm g}=(\widetilde{g}(x^1),\cdots,\widetilde{g}(x^n))^T,\ \ \ \widetilde{g}=\frac{g}{\sqrt{h_{m,d}}}.$$

 Since $\{\eta_{k,d}\}_{k=1}^{\infty}$ is an orthonormal
system in $G_d$, we get
\begin{align*}
\|S_{\rm X}^{m}(g)\|_{G_d}^2=\|\widetilde{c}\|_{2}^2&=
\|((\widetilde{L}_{m})^{*}\widetilde{L}_{m})^{-1}(\widetilde{L}_{m})^{*}\widetilde{\rm g}\|_{2}^2\\
&\leq\|((\widetilde{L}_{m})^{*}\widetilde{L}_{m})^{-1}\|\cdot\|(\widetilde{L}_{m})^{*}\widetilde{\rm g}\|_{2}^{2}\\
&\leq\frac{4}{n^{2}}\|(\widetilde{L}_{m})^{*}\widetilde{\rm g}\|_{2}^{2},
\end{align*}
where $\|\cdot\|_2$ is the Euclidean norm of a vector.  We have
\begin{align*}
\|(\widetilde{L}_{m})^{*}\widetilde{\rm
g}\|_{2}^{2}&=\sum_{k=1}^{m}\Big|\sum_{j=1}^{n}\overline{\widetilde{\eta}_{k,d}(x^{j})}\cdot
\widetilde{g}(x^{j})\Big|^{2}\\
&=\sum_{k=1}^{m}\sum_{j=1}^{n}\sum_{i=1}^{n}\overline{\widetilde{\eta}_{k,d}(x^{j})}
\widetilde{g}(x^{j})\widetilde{\eta}_{k,d}(x^{i})\overline{\widetilde{g}(x^{i})}.
\end{align*}
It follows that
\begin{align*}
J&=\int_{\|\widetilde{H}_{m}-I_{m}\|\leq\frac{1}{2}}\|(\widetilde{L}_{m})^{*}\widetilde{\rm g}\|_{2}^{2}\,d\mu_{m,d}(x^{1})\ldots d\mu_{m,d}(x^{n})\\
&\le\int_{D_d^n}\|(\widetilde{L}_{m})^{*}\widetilde{\rm g}\|_{2}^{2}\,d\mu_{m,d}(x^{1})\ldots d\mu_{m,d}(x^{n})\\
&\le\sum_{k=1}^{m}\sum_{i,j=1}^{n}\int_{D_d^n}\overline{\widetilde{\eta}_{k,d}(x^{j})}
\widetilde{g}(x^{j})\widetilde{\eta}_{k,d}(x^{i})\overline{\widetilde{g}(x^{i})}
\,d\mu_{m,d}(x^{1})\ldots d\mu_{m,d}(x^{n})\\
&=\sum_{k=1}^{m}\sum_{i,j=1}^{n}J_{k,i,j},
\end{align*}
where$$J_{k,i,j}=\int_{D_d^n}\overline{\widetilde{\eta}_{k,d}(x^{j})}
\widetilde{g}(x^{j})\widetilde{\eta}_{k,d}(x^{i})\overline{\widetilde{g}(x^{i})}
\,d\mu_{m,d}(x^{1})\ldots d\mu_{m,d}(x^{n}).$$ If $i\neq j$ and
$1\le k\le m$, then
$$J_{k,i,j}=|\langle\widetilde{g},\widetilde{\eta}_{k,d}\rangle_{H_d}|^2
=|\langle g,\eta_{k,d}\rangle_{G_d}|^2=0;$$
If $i=j$, then
$$J_{k,i,j}=\|\widetilde{\eta}_{k,d}\cdot\widetilde{g}\|_{H_d}^2.$$
Since
$h_{m,d}(x)=\frac{1}{m}\sum\limits_{k=1}^{m}|\eta_{k,d}(x)|^{2}$,
we get
\begin{align*}
J&\le \sum_{k=1}^{m}\sum_{i=j=1}^{n}J_{k,i,j}=n\sum_{k=1}^m\|\widetilde{\eta}_{k,d}\cdot\widetilde{g}\|_{H_d}^2\\
&=n\sum_{k=1}^m\int_{D_d^n}|\widetilde{g}(x)\widetilde{\eta}_{k,d}(x)|^2
\rho_d(x)h_{m,d}(x)\,dx\\
&=n\sum_{k=1}^m\int_{D_d^n}\frac{|g(x)\eta_{k,d}(x)|^2}{h_{m,d}(x)}
\rho_d(x)\,dx\\
&=n\int_{D_d^n}m|g(x)|^2\rho_d(x)\,dx\\
&=nm\cdot\|g\|_{G_d}^2.
\end{align*}
Hence, by \eqref{2.2-0} we have
\begin{align*}&\qquad \int_{\|\widetilde{H}_{m}-I_{m}\|\leq\frac{1}{2}}\|f-S_X^m(f)\|_{G_d}^{2}\,d\mu_{m,d}(x^{1})\ldots
d\mu_{m,d}(x^{n})\\
&\le\|g\|_{G_d}^2+\frac{4}{n^2}J\le(1+\frac{4m}{n})\|g\|_{G_d}^2
\le(1+\frac{4m}{n})(e^{\rm wor}(m,d;\Lz^{\rm all}))^2.
\end{align*}
We conclude that
$$\mathbb{E}\left(\|f- S_{\rm X}^{m}(f)\|_{G_d}^2\ \Big|\ \|\widetilde{H}_{m}-I_{m}\|\le 1/2\right)$$
\begin{align*}
&=\frac{\int_{\|\widetilde{H}_{m}-I_{m}\|\leq\frac{1}{2}}\|f-S_{X}^{m}f\|^{2}_{G_d}
d\mu_{m,d}(x^{1})\ldots d\mu_{m,d}(x^{n})}{\mathbb{P}(\|\widetilde{H}_{m}-I_{m}\|\leq\frac{1}{2})}\\
&\leq\left(1+\frac{4m}{n}\right)\frac{1}{1-\delta}\,(e^{\rm
wor}(m,d;\Lz^{\rm all}))^2,
\end{align*}
where in the last inequality we used \eqref{3.4}. This completes
the proof of Theorem 2.2. $\hfill\Box$

\

\noindent{\it Proof of Theorem 2.3.}\

Applying Theorem 2.1 with $\delta=\frac{1}{2^{\sqrt{2}}}$, we
obtain
\begin{equation}
e^{\rm ran}(n,d;\Lz^{\rm
std})\leq\Big(1+\frac{4m}{n}\Big)^{\frac{1}{2}}\Big(\frac{2^{\sqrt{2}}}{2^{\sqrt{2}}-1}\Big)^{\frac12}e^{\rm
wor}(m,d;\Lz^{\rm all}),\label{4.1}
\end{equation}
where $m,n\in\Bbb N$, and $$
m=\Big\lfloor\frac{n}{48\sqrt{2}\ln(4n)}\Big\rfloor.$$ Since
$1+\frac{4m}{n}\leq1+\frac{1}{12\sqrt{2}\ln(4n)}\leq2$, by
\eqref{4.1} we get
\begin{equation}\label{4.01}e^{\rm ran}(n,d;\Lz^{\rm std})\leq 4e^{\rm wor}(m,d;\Lz^{\rm
all}).\end{equation} It follows that
\begin{align}
n^{\rm ran,\star}(\varepsilon,d;\Lambda^{\rm std})&=\min\big\{n\, \big|\,  e^{\rm ran}(n,d;\Lz^{\rm std})\leq\varepsilon{\rm CRI}_d\big\}\nonumber\\
&\leq\min\big\{n\,\big|\, 4e^{\rm wor}(m,d;\Lz^{\rm all})\leq\varepsilon{\rm CRI}_d\big\}\nonumber\\
&=\min\big\{n\mid e^{\rm wor}(m,d;\Lz^{\rm
all})\leq\frac{\varepsilon}{4}{\rm CRI}_d\big\}.\label{4.2}
\end{align}
We note that
$$m=\Big\lfloor\frac{n}{48\sqrt{2}\ln(4n)}\Big\rfloor\geq\frac{n}{48\sqrt{2}\ln(4n)}-1.$$
This inequality is equivalent to
\begin{equation}
4n\leq192\sqrt{2}(m+1)\ln(4n).\label{4.3}
\end{equation}
Taking logarithm on both sides of \eqref{4.3}, and using the
inequality $\ln x\leq\frac{1}{2}x$ for $x\ge 1$, we get
$$\ln(4n)\leq \ln(m+1)+\ln(192\sqrt{2})+\ln\ln(4n),$$and
$$\frac{1}{2}\ln(4n)\leq\ln(4n)-\ln\ln(4n)\leq\ln(m+1)+\ln(192\sqrt{2}).$$
It follows  from \eqref{4.3} that
\begin{equation}
n\leq96\sqrt{2}(m+1)(\ln(m+1)+\ln(192\sqrt{2})).\label{4.4}
\end{equation}
By \eqref{4.2} and \eqref{4.4} we obtain $$n^{\rm
ran,\star}(\varepsilon,d;\Lambda^{\rm std})\le 96\sqrt2
\Big(n^{\rm wor,\star}(\frac\varepsilon4,d;\Lambda^{\rm
all})+1\Big) \Big(\ln\big(n^{\rm
wor,\star}(\frac\varepsilon4,d;\Lambda^{\rm
all})+1\big)+\ln(192\sqrt{2})\Big),$$proving \eqref{2.14}.

 For sufficiently small $\delta>0$ and
 $m,n\in\mathbb{N}$ satisfying
$$m=\left\lfloor\frac{n}{48(\sqrt{2}\ln(2n)-\ln{\delta})}\right\rfloor,$$
by Theorem 2.2 we have
\begin{align*}
e^{\rm ran}(n,d;\Lz^{\rm std})&\le\Big(1+\frac{4m}{n}\Big)^{\frac12}\frac{1}{\sqrt{1-\delta}}\,e^{\rm wor}(m,d;\Lz^{\rm all})\\
&\le\Big(1+\frac{1}{12\big(\sqrt{2}\ln(2n)+\ln{\frac{1}{\delta}}\big)}\Big)^{\frac12}\frac{1}{\sqrt{1-\delta}}\,e^{\rm wor}(m,d;\Lz^{\rm all})\\
&\le\Big(1+\frac{1}{12\ln{\frac{1}{\delta}}}\Big)^{\frac12}\frac{1}{\sqrt{1-\delta}}\,e^{\rm
wor}(m,d;\Lz^{\rm all})=A_{\delta}\,e^{\rm wor}(m,d;\Lz^{\rm
all}),
\end{align*}
where
$A_{\delta}=\Big(1+\frac{1}{12\ln{\frac{1}{\delta}}}\Big)^{\frac12}\frac{1}{\sqrt{1-\delta}}$.

Using the same method used in the proof of \eqref{4.2}, we have
$$ n^{\rm ran,\star}(\varepsilon,d;\Lambda^{\rm std})\le \min\big\{n\mid e^{\rm wor}(m,d;\Lz^{\rm
all})\leq\frac{\varepsilon}{A_\dz}{\rm CRI}_d\big\}.$$ We note
that
$$n\le48\big(\sqrt{2}\ln(2n)+\ln{\frac{1}{\delta}}\big)(m+1).$$
Taking logarithm on both sides, and using the inequalities $\ln
x\leq\frac{x}{4}$ for $x\ge 9$ and $a+b\le ab$ for $a,b\ge2$, we
get
\begin{align*}
\ln n&\le \ln48+\ln\big(\sqrt{2}\ln(2n)+\ln{\frac{1}{\delta}}\big)+\ln(m+1)\\
&\le \ln48+\ln(\sqrt{2}\ln(2n))+\ln\ln{\frac{1}{\delta}}+\ln(m+1)\\
&\le \ln48+\frac{\sqrt{2}}{4}\ln(2n)+\ln\ln{\frac{1}{\delta}}+\ln(m+1).
\end{align*}
Since
$$\frac{\sqrt{2}}{4}\ln(2n)\le \ln n-\frac{\sqrt{2}}{4}\ln(2n)\ \ \text{for}\ \ n\ge9,$$
we get
$$\sqrt{2}\ln(2n)\le 4\big(\ln48+\ln\ln{\frac{1}{\delta}}+\ln(m+1)\big).$$
It follows that
$$n\le48\big(4\big(\ln48+\ln\ln{\frac{1}{\delta}}+\ln(m+1)\big)+\ln{\frac{1}{\delta}}\big)(m+1).$$
We conclude that for sufficiently small $\delta>0$,
\begin{align}\label{4.4-1}
n^{\rm ran,\star}(\varepsilon,d;\Lambda^{\rm std}) \le
48\Big(4\big(\ln48&+\ln\ln{\frac{1}{\delta}}+ \ln\big(n^{\rm
wor,\star}(\frac{\varepsilon}{A_\delta},d;\Lambda^{\rm
all})+1\big)\big)\\ &+\ln{\frac{1}{\delta}}\Big) \big(n^{\rm
wor,\star}(\frac{\varepsilon}{A_\delta},d;\Lambda^{\rm
all})+1\big),\nonumber
\end{align}proving \eqref{2.15}. Theorem 2.3 is proved.  $\hfill\Box$

\section{Equivalence results of algebraic tractability}

First we   consider the equivalences of  ALG-PT and ALG-SPT for
$\Lambda^{\rm std}$ and $\Lambda^{\rm all}$ in the randomized
setting.    The equivalent results for the normalized error
criterion can be found in \cite[Theorem 22.19]{NW3}.  For the
absolute error criterion, \cite[Theorem 22.20]{NW3} shows the
equivalence of ALG-PT under the condition
\begin{equation}\label{4.11}\lz_{1,d}\le C_\lz d^{s_\lz} \ \ {\rm for \ all}\ d\in\Bbb N, \ {\rm  some}\ C_\lz>0,\ {\rm and  \ some}\ s_\lz\ge 0,\end{equation}
and the equivalence of ALG-SPT under the condition \eqref{4.11}
with $s_\lz=0$.

 We obtain the following  equivalent
results of ALG-PT and  ALG-SPT  without any condition. Hence, the
condition \eqref{4.11} is unnecessary. This solves Open Problem
101 as posed by Novak and Wo\'zniakowski in  \cite{NW3}.

\begin{thm}
We consider the problem ${\rm APP}=\{{\rm APP}_d\}_{d\in \Bbb N}$
in the randomized  setting for the absolute error criterion. Then,

$\bullet$   ${\rm ALG}$-${\rm PT}$ for $\Lambda^{\rm all}$  is
equivalent to ${\rm ALG}$-${\rm PT}$
 for $\Lambda^{\rm std}$ .

$\bullet$  ${\rm ALG}$-${\rm SPT}$  for $\Lambda^{\rm all}$ is
equivalent to ${\rm ALG}$-${\rm SPT}$  for $\Lambda^{\rm std}$. In
this case,  the exponents of ${\rm ALG}$-${\rm SPT}$ for
$\Lambda^{\rm all}$ and  $\Lambda^{\rm std}$ are the same.
\end{thm}

\begin{proof} It follows from \eqref{2.4-1} that ALG-PT (ALG-SPT) for $\Lambda^{\rm std}$ means ALG-PT (ALG-SPT) for  $\Lambda^{\rm all}$ in the randomized  setting.
Since ALG-PT (ALG-SPT) for $\Lambda^{\rm all}$ in the worst case
setting is equivalent to ${\rm ALG}$-${\rm PT}$ (ALG-SPT) for
$\Lambda^{\rm all}$
 in the randomized  setting, it suffices to show that ALG-PT (ALG-SPT) for $\Lambda^{\rm all}$ in the worst case
setting means that ALG-PT (ALG-SPT) for  $\Lambda^{\rm std}$ in
the randomized  setting.

Suppose that ALG-PT  holds for $\Lambda^{\rm all}$ in the worst
case setting. Then there exist $ C\ge1 $ and non-negative $ p,q$
such that \begin{equation}n^{\rm
wor,ABS}(\varepsilon,d;\Lambda^{\rm all})\leq
Cd^{q}\varepsilon^{-p},\ \  \text{for all}\ \ d\in\mathbb{N},\
\varepsilon\in(0,1).\label{4.6-0}\end{equation} It follows from
\eqref{2.17} and \eqref{4.6-0} that
\begin{align*}
n^{\rm ran, ABS}(\varepsilon,d;\Lambda^{\rm
std})&\leq C_\oz \(Cd^{q}(\frac{\varepsilon}{4})^{-p}+1\)^{1+\oz}\\
&\le C_\oz (2C\, 4^{p})^{1+\oz}d^{q(1+\oz)}\vz^{-p(1+\oz)},
\end{align*}which means that ALG-PT  holds for  $\Lambda^{\rm std}$ in
the randomized  setting.

If ALG-SPT holds  for $\Lambda^{\rm all}$ in the worst case
setting, then \eqref{4.6-0} holds with $q=0$. Using the same
method we obtain
$$ n^{\rm ran, ABS}(\varepsilon,d;\Lambda^{\rm
std})\le C_\oz(2C\, 4^{p})^{1+\oz}\vz^{-p(1+\oz)},$$which means
that ALG-SPT  holds for  $\Lambda^{\rm std}$ in the randomized
setting. Furthermore, since $\oz$ can be arbitrary small, by
Corollary 2.1 we have
\begin{align*} {\rm ALG\!-\!}p^{\rm ran, ABS}(\Lz^{\rm std})&\le
{\rm ALG\!-\!}p^{\rm wor, ABS}(\Lz^{\rm all})\\  ={\rm
ALG\!-\!}p^{\rm ran, ABS}(\Lz^{\rm all})&\le {\rm ALG\!-\!}p^{\rm
ran, ABS}(\Lz^{\rm std}),
\end{align*} which means that the exponents of ${\rm ALG}$-${\rm
SPT}$ for $\Lambda^{\rm all}$ and  $\Lambda^{\rm std}$ are the
same. This completes the proof of Theorem 4.1.
\end{proof}

Next we consider the equivalence of  ALG-QPT for $\Lambda^{\rm
std}$ and $\Lambda^{\rm all}$ in the randomized setting.
 The result for the
normalized error criterion can be found in \cite[Theorem
22.21]{NW3}. For the absolute error criterion, \cite[Theorem
22.22]{NW3} shows the equivalence of ALG-QPT under the condition
\begin{equation}\label{4.12}\underset{d\to\infty}{\lim \sup}\ \lz_{1,d}<\infty.\end{equation}

 We obtain the following  equivalent
result of ALG-QPT without any condition. Hence, the condition
\eqref{4.12} is unnecessary.  This solves Open Problem 102 as
posed by Novak and Wo\'zniakowski in  \cite{NW3}.

\begin{thm}
We consider the problem $\rm APP=\{APP_d\}_{d\in\mathbb{N}}$ in the randomized setting for the absolute error criterion.
 Then, ${\rm ALG}$-${\rm QPT}$ for $\Lambda^{\rm all}$  is
equivalent to ${\rm ALG}$-${\rm QPT}$
 for $\Lambda^{\rm std}$. In this case,  the exponents
 of ${\rm ALG}$-${\rm QPT}$ for
$\Lambda^{\rm all}$ and  $\Lambda^{\rm std}$ are the same.
\end{thm}

\begin{proof}Similar to the proof of Theorem 4.1,
 it is enough to prove that ALG-QPT for
$\Lambda^{\rm all}$ in the worst case setting implies ALG-QPT for
 $\Lambda^{\rm std}$ in the randomized setting.

Suppose that  ALG-QPT holds for $\Lambda^{\rm all}$ in the worst
case setting. Then  there exist $ C\ge 1$ and non-negative $t$
such that \begin{equation}\label{4.9}n^{\rm
wor,ABS}(\varepsilon,d;\Lambda^{\rm all})\leq C
\exp(t(1+\ln{d})(1+\ln{\varepsilon^{-1}})),\ \text{for all}\
d\in\mathbb{N},\ \varepsilon\in(0,1).\end{equation} It follows
from \eqref{2.17} and \eqref{4.9} that for $\oz>0$,
\begin{align*}n^{\rm ran,ABS}(\varepsilon,d;\Lambda^{\rm std})&\le
C_\oz \( n^{\rm wor,ABS}(\varepsilon/4,d;\Lambda^{\rm
all})+1\)^{1+\oz}
\\
&\leq C_\oz\(C
\exp\big(t(1+\ln{d})\big(1+\ln\big(\frac{\varepsilon}{4}\big)^{-1})\big)+1\)^{1+\oz}
\\ &\leq C_\oz (2C)^{1+\oz}\exp\big(
(1+\oz)t(1+\ln{d})(1+\ln4+\ln\varepsilon^{-1})\big)\\ &\le C_\oz
(2C)^{1+\oz}\exp\big( t^*(1+\ln{d})(1+\ln\varepsilon^{-1})\big)
,\end{align*}where $t^*=(1+\oz)(1+\ln 4)t$.  This implies that
ALG-QPT holds for
 $\Lambda^{\rm std}$ in the randomized setting.

Next we show that  the exponents ALG-$t^{\rm ran, ABS}(\Lz^{\rm
all})$ and ALG-$t^{\rm ran, ABS}(\Lz^{\rm std})$ are equal if
ALG-QPT holds for $\Lambda^{\rm all}$ in the worst case setting.
We have
\begin{align*}  {\rm ALG\!-\!}t^{\rm wor, ABS}(\Lz^{\rm
all})  ={\rm ALG\!-\!}t^{\rm ran, ABS}(\Lz^{\rm all})\le {\rm
ALG\!-\!}t^{\rm ran, ABS}(\Lz^{\rm std}).\end{align*} It suffices
to show that
$${\rm
ALG\!-\!}t^{\rm ran, ABS}(\Lz^{\rm std})\le {\rm ALG\!-\!}t^{\rm
wor, ABS}(\Lz^{\rm all}) .$$ Note that using \eqref{2.17} we can
only obtain that
$${\rm
ALG\!-\!}t^{\rm ran, ABS}(\Lz^{\rm std})\le (1+\ln 4)\cdot {\rm
ALG\!-\!}t^{\rm wor, ABS}(\Lz^{\rm all}) .$$ Instead we use
\eqref{2.18}. For sufficiently small $\delta>0$ and $\oz>0$,  it
follows from \eqref{2.18} and \eqref{4.9} that
\begin{align*}
n^{\rm ran,ABS}(\varepsilon,d;\Lambda^{\rm std})& \le C_{\oz,\dz}
 \big(n^{\rm
wor,ABS}(\frac{\varepsilon}{A_\delta},d;\Lambda^{\rm
all})+1\big)^{1+\oz}\\ &\le C_{\oz,\dz}(2C)^{1+\oz}\exp\big(
(1+\oz)t(1+\ln{d})(1+\ln A_\dz +\ln\varepsilon^{-1})\big)\\&\le
C_{\oz,\dz}(2C)^{1+\oz}\exp\big( (1+\oz)t(1+\ln
A_\dz)(1+\ln{d})(1+\ln\varepsilon^{-1})\big),
\end{align*}where
$A_{\delta}=\big(1+\frac{1}{12\ln{\frac{1}{\delta}}}\big)^{\frac12}\frac{1}{\sqrt{1-\delta}}$.
Taking the infimum over $t$ for which \eqref{4.9} holds, and
noting that $ \lim\limits_{(\dz,\oz)\to (0,0)}(1+\oz)(1+\ln
A_\dz)=1$,
 we get that
\begin{align*}{\rm ALG\!-\!}t^{\rm ran, ABS}(\Lz^{\rm std})\le {\rm ALG\!-\!}t^{\rm
wor, ABS}(\Lz^{\rm all}).
\end{align*}This completes the proof of Theorem 4.2.
\end{proof}

Now we   consider the equivalence of  ALG-WT for $\Lambda^{\rm
std}$ and $\Lambda^{\rm all}$ in the randomized setting.
 The result for the normalized error criterion can be
found in \cite[Theorem 22.5]{NW3}. For the absolute error
criterion, \cite[Theorem 22.6]{NW3} shows the equivalence of
ALG-WT under the condition
\begin{equation}\label{4.13}\lim_{d\to\infty}\frac{\ln\max(\lz_{1,d},1)}{d}=0.\end{equation}

We obtain the following  equivalent result of ALG-WT without any
condition. Hence, the condition \eqref{4.13} is unnecessary.  This
solves Open Problem 98 as posed by Novak and Wo\'zniakowski in
\cite{NW3}.

\begin{thm}
We consider the problem $\rm APP=\{APP_d\}_{d\in\mathbb{N}}$ in
the randomized setting for the absolute error criterion. Then,
${\rm ALG}$-${\rm WT}$ for $\Lambda^{\rm all}$  is equivalent to
${\rm ALG}$-${\rm WT}$
 for $\Lambda^{\rm std}$.
\end{thm}
\begin{proof} The proof is identical to the proof of Theorem 4.4 with $s = t = 1$ for the absolute error criterion. We omit the details. \end{proof}

Finally, we   consider the equivalences of ALG-$(s, t)$-WT and
ALG-UWT for $\Lambda^{\rm std}$ and $\Lambda^{\rm all}$ in the
randomized setting. As far as we know, these equivalences have not
been studied yet. We obtain the following equivalent results of
ALG-$(s, t)$-WT and ALG-UWT for the absolute  or  normalized error
criterion without any condition.

\begin{thm}
We consider the problem $\rm APP=\{APP_d\}_{d\in\mathbb{N}}$ in
the randomized setting for the absolute or  normalized error
criterion. Then for
 fixed $s,t>0$, ${\rm ALG}$-$(s,t)$-${\rm WT}$ for $\Lambda^{\rm all}$  is
equivalent to ${\rm ALG}$-$(s,t)$-${\rm WT}$
 for $\Lambda^{\rm std}$.
\end{thm}

\begin{proof}
Again
 it is enough to prove that ${\rm ALG}$-$(s,t)$-${\rm WT}$ for
$\Lambda^{\rm all}$ in the worst case setting implies ${\rm ALG}$-$(s,t)$-${\rm WT}$ for
$\Lambda^{\rm std}$ in the randomized setting.

Suppose that ${\rm ALG}$-$(s,t)$-${\rm WT}$ holds for
$\Lambda^{\rm all}$ in the worst case setting. Then  we have  for
$\star\in\{{\rm ABS,\,NOR}\}$,
\begin{equation}
\lim_{\varepsilon^{-1}+d\rightarrow\infty}\frac{\ln n^{\rm
wor,\star}(\varepsilon,d;\Lambda^{\rm
all})}{\varepsilon^{-s}+d^{t}}=0.\label{4.10}
\end{equation}
It follows from \eqref{2.17}  that for $\oz>0$,
\begin{align*}
\frac{\ln n^{\rm ran,\star}(\varepsilon,d;\Lambda^{\rm
std})}{\varepsilon^{-s}+d^{t}} &\leq\frac{\ln
\Big(C_\oz\big(n^{\rm ran,\star}(\varepsilon/4,d;\Lambda^{\rm
all})+1\big)^{1+\oz}\Big)}{\varepsilon^{-s}+d^{t}}\\ &\le
\frac{\ln
(C_\oz2^{1+\oz})}{\varepsilon^{-s}+d^{t}}+\frac{4^s(1+\oz)\,\ln
n^{\rm wor,\star}(\varepsilon/4,d;\Lambda^{\rm
all})}{(\varepsilon/4)^{-s}+d^{t}}.
\end{align*}
Since $\varepsilon^{-1}+d\rightarrow\infty$ is equivalent to
$\vz^{-s}+d^t\to \infty$,  by \eqref{4.10} we get that
$$\lim\limits_{\varepsilon^{-1}+d\rightarrow\infty}\frac{\ln (C_\oz2^{1+\oz})}{\varepsilon^{-s}+d^{t}}=0\ \ \ {\rm and} \ \
\lim_{\varepsilon^{-1}+d\rightarrow\infty}\frac{\ln n^{\rm
wor,\star}(\varepsilon/4,d;\Lambda^{\rm
all})}{(\varepsilon/4)^{-s}+d^{t}}=0.$$We  obtain
$$\lim\limits_{\varepsilon^{-1}+d\rightarrow\infty} \frac{\ln n^{\rm ran,\star}(\varepsilon,d;\Lambda^{\rm
std})}{\varepsilon^{-s}+d^{t}}=0,$$ which implies that
 ${\rm ALG}$-$(s,t)$-${\rm WT}$ holds for
$\Lambda^{\rm std}$ in the randomized setting. This completes the
proof of
 Theorem 4.4.
\end{proof}

\begin{thm}
We consider the problem $\rm APP=\{APP_d\}_{d\in\mathbb{N}}$ in
the randomized setting for the absolute or normalized error
criterion.  Then, ${\rm ALG}$-${\rm UWT}$ for $\Lambda^{\rm all}$
is equivalent to ${\rm ALG}$-${\rm UWT}$
 for $\Lambda^{\rm std}$.
\end{thm}

\begin{proof}
By definition we know  that ${\rm APP}$ is ALG-UWT if and only if
 ${\rm APP}$ is ALG-$(s,t)$-WT for all $s,t>0$. Then Theorem 4.5 follows from Theorem 4.4 immediately.
\end{proof}

\noindent{\it Proof of Theorem 2.5.} \

Theorem 2.5 follows from Theorems 4.1-4.3 immediately.
$\hfill\Box$

\section{Equivalence results of exponential  tractability}

First we consider exponential convergence. Assume that there exist
two constants $A\ge1$ and $q\in(0,1)$ such that
\begin{equation}\label{5.01}e^{\rm
wor}(n,d;\Lambda^{\rm all})\le
Aq^{n+1}\sqrt{\lz_{1,d}}\,.\end{equation} Novak and Wo\'zniakowski
proved in \cite[Theorem 22.18]{NW3} that there exist two constants
$C_1\ge1$ and $q_1\in (q,1)$ independent of $d$ and $n$ such that
\begin{equation}\label{5.02} e^{\rm
ran} (n,d;\Lambda^{\rm std})\le C_1 A\, q_1^{\sqrt
{n}}\,\sqrt{\lz_{1,d}}\,.\end{equation} If  $A, q$ in \eqref{5.01}
are independent of $d$, then $$n^{\rm
wor,NOR}(\varepsilon,d;\Lambda^{\rm all})\leq
C_2(\ln\varepsilon^{-1}+1),$$ and $$n^{\rm
ran,NOR}(\varepsilon,d;\Lambda^{\rm std})\leq
C_3(\ln\varepsilon^{-1}+1)^2.$$  Novak and Wo\'zniakowski posed
Open Problem 100 which states

 (1) Verify if the upper bound in \eqref{5.02} can be
improved.

 (2) Find the smallest $p$ for which there holds
$$n^{\rm
ran,NOR}(\varepsilon,d;\Lambda^{\rm std})\leq
C_4(\ln\varepsilon^{-1}+1)^p.$$ We know that $p\le 2$, and if
\eqref{5.01} is sharp then $p\ge1$.

The following theorem  gives a confirmative solution to  Open
Problem 100 (1). We improve enormously the upper bound $q_1^{\sqrt
n}$ in \eqref{5.02} to $q_2^{\frac{n}{\ln(4n)}}$ in \eqref{5.05},
where $q_1,q_2\in(q,1)$.

\begin{thm} Let $m,n\in\Bbb N$ and \begin{equation}\label{5.03}
m=\Big\lfloor\frac{n}{48\sqrt{2}\ln(4n)}\Big\rfloor.\end{equation}
Then we have \begin{equation}\label{5.04}e^{\rm ran}(n,d;\Lz^{\rm
std})\leq 4e^{\rm wor}(m,d;\Lz^{\rm all}).\end{equation}
 Specifically, if \eqref{5.01} holds, then we have
  \begin{equation}\label{5.05}e^{\rm ran}(n,d;\Lz^{\rm
std})\leq 4 A q_2^{\frac{n}{\ln(4n)}} \sqrt{\lz_{1,d}}\,
,\end{equation}where $q_2=q^{\frac1{48\sqrt2}}\in(q,1)$.
\end{thm}
\begin{proof} Inequality \eqref{5.04} is just \eqref{4.01}, which
has been  proved. If \eqref{5.01} holds, then by \eqref{5.03} and
\eqref{5.04} we get
$$ e^{\rm ran}(n,d;\Lz^{\rm
std})\leq 4 A\,
q^{\big\lfloor\frac{n}{48\sqrt{2}\ln(4n)}\big\rfloor+1}
\sqrt{\lz_{1,d}}\,\le 4 A q^{\frac{n}{48\sqrt{2}\ln(4n)}}
\sqrt{\lz_{1,d}}=4 A q_2^{\frac{n}{\ln(4n)}} \sqrt{\lz_{1,d}}.$$
This completes the proof of Theorem 5.1.
\end{proof}

Now  we   consider the equivalences of various notions of
exponential tractability  for $\Lambda^{\rm std}$ and
$\Lambda^{\rm all}$ in the randomized setting.   As far as we
know, there is hardly any result for these  equivalences.

First we   consider the equivalences of  EXP-PT and EXP-SPT for
$\Lambda^{\rm std}$ and $\Lambda^{\rm all}$ in the randomized
setting.
 We obtain the following  equivalent
results of ALG-PT and  ALG-SPT  without any condition.

\begin{thm}
We consider the problem $\rm APP=\{APP_d\}_{d\in\mathbb{N}}$ in
the randomized setting for the absolute or normalized error
criterion.  Then,

$\bullet$   ${\rm EXP}$-${\rm PT}$ for $\Lambda^{\rm all}$  is
equivalent to ${\rm EXP}$-${\rm PT}$
 for $\Lambda^{\rm std}$ .

$\bullet$  ${\rm EXP}$-${\rm SPT}$  for $\Lambda^{\rm all}$ is
equivalent to ${\rm EXP}$-${\rm SPT}$  for $\Lambda^{\rm std}$. In
this case,  the exponents of ${\rm EXP}$-${\rm SPT}$ for
$\Lambda^{\rm all}$ and  $\Lambda^{\rm std}$ are the same.
\end{thm}
\begin{proof}
Again, it is enough to prove that EXP-PT for $\Lambda^{\rm all}$
in the worst case setting implies EXP-PT for $\Lambda^{\rm std}$
in the randomized setting.

Suppose that EXP-PT  holds for $\Lambda^{\rm all}$ in the worst
case setting. Then there exist $C\ge1 $ and non-negative $ p,q$,
for $\star\in\{{\rm ABS,\,NOR}\}$ such that
\begin{equation}\label{5.1}n^{\rm
wor,\star}(\varepsilon,d;\Lambda^{\rm all})\leq
Cd^{q}(\ln\varepsilon^{-1}+1)^{p},\ \text{for all}\
d\in\mathbb{N},\ \varepsilon\in(0,1).\end{equation}
 It follows from
\eqref{2.17} and \eqref{5.1} that
\begin{align*}
n^{\rm ran, \star}(\varepsilon,d;\Lambda^{\rm
std})&\leq C_\oz \(Cd^{q}(\ln(\frac{\varepsilon}{4})^{-1}+1)^p+1\)^{1+\oz}\\
&\le C_\oz (2C)^{1+\oz} (1+\ln
4)^{p(1+\oz)}d^{q(1+\oz)}(\ln\varepsilon^{-1}+1)^{p(1+\oz)},
\end{align*}which means that EXP-PT  holds for  $\Lambda^{\rm std}$ in
the randomized  setting.

If EXP-SPT holds  for $\Lambda^{\rm all}$ in the worst case
setting, then \eqref{5.1} holds with $q=0$. We obtain
$$ n^{\rm ran, \star}(\varepsilon,d;\Lambda^{\rm
std})\le C_\oz (2C)^{1+\oz} (1+\ln
4)^{p(1+\oz)}(\ln\varepsilon^{-1}+1)^{p(1+\oz)},$$which means that
EXP-SPT  holds for  $\Lambda^{\rm std}$ in the randomized setting.
Furthermore, in this case we have \begin{align*} {\rm
EXP\!-\!}p^{\rm ran, \star}(\Lz^{\rm std})&\le  {\rm
EXP\!-\!}p^{\rm wor, \star}(\Lz^{\rm all})\\  ={\rm
EXP\!-\!}p^{\rm ran, \star}(\Lz^{\rm all})&\le {\rm
EXP\!-\!}p^{\rm ran, \star}(\Lz^{\rm std}),
\end{align*} which means that the exponents of ${\rm EXP}$-${\rm
SPT}$ for $\Lambda^{\rm all}$ and  $\Lambda^{\rm std}$ are the
same. This completes the proof of Theorem 5.2.
\end{proof}

\begin{rem} We remark that if \eqref{5.01} holds with $A, q$ independent of $d$, then the problem {\rm APP} is {\rm
EXP-SPT} for $\Lambda^{\rm all}$ in the randomized setting for the
normalized error criterion, and the exponent ${\rm EXP\!-\!}p^{\rm
wor, NOR}(\Lz^{\rm all})\le 1$. If \eqref{5.01} is sharp, then
${\rm EXP\!-\!}p^{\rm wor, NOR}(\Lz^{\rm all})= 1$.

  Open Problem 100 (2) is equivalent to finding the exponent $ {\rm
EXP\!-\!}p^{\rm ran, NOR}(\Lz^{\rm std})$ of ${\rm EXP}$-${\rm
SPT}$. By Theorem 5.2 we obtain that if \eqref{5.01} holds, then $
{\rm EXP\!-\!}p^{\rm ran, NOR}(\Lz^{\rm std})\le 1,$ and if
\eqref{5.01} is sharp, then $ {\rm EXP\!-\!}p^{\rm ran,
NOR}(\Lz^{\rm std})=1$.

This solves Open Problem 100 (2) as posed by Novak and
Wo\'zniakowski in  \cite{NW3}.
\end{rem}

Next we consider the equivalence of  EXP-QPT for $\Lambda^{\rm
std}$ and $\Lambda^{\rm all}$ in the randomized setting.
 We obtain the following  equivalent
result of EXP-QPT without any condition.

\begin{thm}
We consider the problem $\rm APP=\{APP_d\}_{d\in\mathbb{N}}$ in
the randomized setting for the absolute or normalized error
criterion.
 Then, ${\rm EXP}$-${\rm QPT}$ for $\Lambda^{\rm all}$  is
equivalent to ${\rm EXP}$-${\rm QPT}$
 for $\Lambda^{\rm std}$. In this case,  the exponents
 of ${\rm EXP}$-${\rm QPT}$ for
$\Lambda^{\rm all}$ and  $\Lambda^{\rm std}$ are the same.
\end{thm}

\begin{proof}Again,
 it is enough to prove that EXP-QPT for
$\Lambda^{\rm all}$ in the worst case setting implies EXP-QPT for
 $\Lambda^{\rm std}$ in the randomized setting.

Suppose that  EXP-QPT holds for $\Lambda^{\rm all}$ in the worst
case setting. Then  there exist $ C\ge 1$ and non-negative $t$
such that \begin{equation}\label{5.2}n^{\rm wor,\star}(\va
,d;\Lz)\leq C \exp(t(1+\ln{d})(1+\ln(\ln\varepsilon^{-1}+1))),\
\text{for all}\ d\in\mathbb{N},\
\varepsilon\in(0,1).\end{equation} It follows from \eqref{2.17}
and \eqref{5.2} that for $\oz>0$,
\begin{align*} &\quad\  n^{\rm ran,\star}(\varepsilon,d;\Lambda^{\rm std})\\ &\le
C_\oz \( n^{\rm wor,\star}(\varepsilon/4,d;\Lambda^{\rm
all})+1\)^{1+\oz}
\\
&\leq C_\oz\(C
\exp\big(t(1+\ln{d})\big(1+\ln(\ln\varepsilon^{-1}+\ln
4+1))\big)+1\)^{1+\oz}
\\ &\leq C_\oz (2C)^{1+\oz}\exp\big(
(1+\oz)t(1+\ln{d})(1+\ln(\ln4+1)+ \ln(\ln\varepsilon^{-1}+1))\big)\\
&\le C_\oz (2C)^{1+\oz}\exp\big(
t^*(1+\ln{d})(1+\ln(\ln\varepsilon^{-1}+1))\big)
,\end{align*}where $t^*=(1+\oz)(1+\ln(\ln 4+1))t$, in the third
inequality we used the fact $$\ln (1+a+b)\le \ln(1+a)+\ln (1+b),\
\ \ a,b\ge 0.$$ This implies that EXP-QPT holds for
 $\Lambda^{\rm std}$ in the randomized setting.

Next we show that  the exponents EXP-$t^{\rm ran, \star}(\Lz^{\rm
all})$ and EXP-$t^{\rm ran, \star}(\Lz^{\rm std})$ are equal if
EXP-QPT holds for $\Lambda^{\rm all}$ in the worst case setting.
We have
\begin{align*}  {\rm EXP\!-\!}t^{\rm wor, \star}(\Lz^{\rm
all})  ={\rm EXP\!-\!}t^{\rm ran, \star}(\Lz^{\rm all})\le {\rm
EXP\!-\!}t^{\rm ran, \star}(\Lz^{\rm std}).\end{align*} It suffices
to show that
$${\rm
EXP\!-\!}t^{\rm ran, \star}(\Lz^{\rm std})\le {\rm EXP\!-\!}t^{\rm
wor, \star}(\Lz^{\rm all}) .$$ Note that using \eqref{2.17} we can
only obtain that
$${\rm
EXP\!-\!}t^{\rm ran, \star}(\Lz^{\rm std})\le (1+\ln 4)\cdot {\rm
EXP\!-\!}t^{\rm wor, \star}(\Lz^{\rm all}) .$$ Instead we use
\eqref{2.18}. For sufficiently small $\delta>0$ and $\oz>0$,  it
follows from \eqref{2.18} and \eqref{5.2} that
\begin{align*}
&\quad \ n^{\rm ran,\star}(\varepsilon,d;\Lambda^{\rm std})\\ &
\le C_{\oz,\dz}
 \big(n^{\rm
wor,\star}(\frac{\varepsilon}{A_\delta},d;\Lambda^{\rm
all})+1\big)^{1+\oz}\\ &\le C_{\oz,\dz}(2C)^{1+\oz}\exp\big(
(1+\oz)t(1+\ln{d})(1+\ln (\ln A_\dz+1)
+\ln(\ln\varepsilon^{-1}+1))\big)\\&\le
C_{\oz,\dz}(2C)^{1+\oz}\exp\big( (1+\oz)t(1+\ln(\ln
A_\dz+1))(1+\ln{d})(1+\ln(\ln\varepsilon^{-1}+1))\big),
\end{align*}where
$A_{\delta}=\big(1+\frac{1}{12\ln{\frac{1}{\delta}}}\big)^{\frac12}\frac{1}{\sqrt{1-\delta}}$.
Taking the infimum over $t$ for which \eqref{5.2} holds, and
noting that $ \lim\limits_{(\dz,\oz)\to (0,0)}(1+\oz)(1+\ln(\ln
A_\dz+1))=1$,
 we get that
\begin{align*}{\rm EXP\!-\!}t^{\rm ran, \star}(\Lz^{\rm std})\le {\rm EXP\!-\!}t^{\rm
wor, \star}(\Lz^{\rm all}).
\end{align*}This completes the proof of Theorem 5.4.
\end{proof}

Finally, we   consider the equivalences of EXP-$(s, t)$-WT
(including EXP-WT) and EXP-UWT for $\Lambda^{\rm std}$ and
$\Lambda^{\rm all}$ in the randomized setting.  We obtain the
following equivalent results of EXP-$(s, t)$-WT  (including
EXP-WT) and EXP-UWT for the absolute or normalized error criterion
without any condition.

\begin{thm}
We consider the problem $\rm APP=\{APP_d\}_{d\in\mathbb{N}}$ in
the randomized setting for the absolute or normalized error
criterion. Then for
 fixed $s,t>0$, ${\rm EXP}$-$(s,t)$-${\rm WT}$ for $\Lambda^{\rm all}$  is
equivalent to ${\rm EXP}$-$(s,t)$-${\rm WT}$
 for $\Lambda^{\rm std}$. Specifically, ${\rm EXP}$-${\rm WT}$ for $\Lambda^{\rm all}$  is
equivalent to ${\rm EXP}$-${\rm WT}$
 for $\Lambda^{\rm std}$.
\end{thm}

\begin{proof}
Again,
 it is enough to prove that ${\rm EXP}$-$(s,t)$-${\rm WT}$ for
$\Lambda^{\rm all}$ in the worst case setting implies ${\rm EXP}$-$(s,t)$-${\rm WT}$ for
$\Lambda^{\rm std}$ in the randomized setting.

Suppose that ${\rm EXP}$-$(s,t)$-${\rm WT}$ holds for
$\Lambda^{\rm all}$ in the worst case setting. Then  we have  for
$\star\in\{{\rm ABS,\,NOR}\}$,
\begin{equation}
\label{5.3}\lim_{\varepsilon^{-1}+d\rightarrow\infty}\frac{\ln
n^{\rm wor,\star}(\varepsilon,d;\Lambda^{\rm
all})}{(1+\ln\varepsilon^{-1})^{s}+d^{t}}=0.
\end{equation}
It follows from \eqref{2.17}  that for $\oz>0$,
\begin{align*}
&\quad\ \frac{\ln n^{\rm ran,\star}(\varepsilon,d;\Lambda^{\rm
std})}{(1+\ln\varepsilon^{-1})^{s}+d^{t}} \leq\frac{\ln
\Big(C_\oz\big(n^{\rm ran,\star}(\varepsilon/4,d;\Lambda^{\rm
all})+1\big)^{1+\oz}\Big)}{(1+\ln\varepsilon^{-1})^{s}+d^{t}}\\
&\le \frac{\ln
(C_\oz2^{1+\oz})}{(1+\ln\varepsilon^{-1})^{s}+d^{t}}+\frac{(1+\ln
4)^s(1+\oz)\,\ln n^{\rm wor,\star}(\varepsilon/4,d;\Lambda^{\rm
all})}{(1+\ln(\varepsilon/4)^{-1})^{s}+d^{t}}.
\end{align*}
Since $\varepsilon^{-1}+d\rightarrow\infty$ is equivalent to
$(1+\ln\varepsilon^{-1})^{s}+d^{t}\to \infty$,  by \eqref{5.3} we
get that
$$\lim\limits_{\varepsilon^{-1}+d\rightarrow\infty}\frac{\ln (C_\oz2^{1+\oz})}{(1+\ln\varepsilon^{-1})^{s}+d^{t}}=0\ \ \ {\rm and} \ \
\lim_{\varepsilon^{-1}+d\rightarrow\infty}\frac{\ln n^{\rm
wor,\star}(\varepsilon/4,d;\Lambda^{\rm
all})}{(1+\ln(\varepsilon/4)^{-1})^{s}+d^{t}}=0.$$We  obtain
$$\lim\limits_{\varepsilon^{-1}+d\rightarrow\infty} \frac{\ln n^{\rm ran,\star}(\varepsilon,d;\Lambda^{\rm
std})}{(\ln\varepsilon^{-1})^{s}+d^{t}}=0,$$ which implies that
 ${\rm EXP}$-$(s,t)$-${\rm WT}$ holds for
$\Lambda^{\rm std}$ in the randomized setting.

Specifically, EXP-WT is just  ${\rm EXP}$-$(s,t)$-${\rm WT}$ with
$s=t=1$.

This completes the proof of
 Theorem 5.5.
\end{proof}

\begin{thm}
We consider the problem $\rm APP=\{APP_d\}_{d\in\mathbb{N}}$ in
the randomized setting for the absolute or normalized error
criterion.  Then, ${\rm EXP}$-${\rm UWT}$ for $\Lambda^{\rm all}$
is equivalent to ${\rm EXP}$-${\rm UWT}$
 for $\Lambda^{\rm std}$.
\end{thm}

\begin{proof}
By definition we know  that ${\rm APP}$ is EXP-UWT if and only if
 ${\rm APP}$ is EXP-$(s,t)$-WT for all $s,t>0$. Then Theorem 5.6 follows from Theorem 5.5 immediately.
\end{proof}

\noindent{\it Proof of Theorem 2.6.}  \

Theorem 2.6 follows from Theorems 4.4, 4.5, 5.2, and 5.4-5.6
immediately.  $\hfill\Box$

\

 \noindent{\bf Acknowledgment}  This work was supported by the
National Natural Science Foundation of China (Project no.
11671271).

\end{document}